\documentclass[
  oneside,
  11pt, a4paper,
  footinclude=true,
  headinclude=true,
  cleardoublepage=empty
]{scrbook}

\usepackage{lipsum}
\usepackage[linedheaders,parts,pdfspacing]{classicthesis}
\usepackage{amsmath}
\usepackage{amsthm}
\usepackage{acronym}
\usepackage{tikz}
\usetikzlibrary{arrows,chains,matrix,positioning,scopes}
\usepackage[mathscr]{euscript}
\usepackage{tikz}
\usepackage{amssymb}
\usepackage{mathrsfs}
\usepackage{graphicx}
\usepackage{parskip}

\newcommand{\adj}{\rotatebox[origin=c]{180}{\ensuremath\vdash}}

\title{Noncommutative Algebra and Noncommutative Geometry}
\author{Anastasis Kratsios}
\begin{document}
\maketitle

\pdfbookmark[0]{Foreword}{Foreword}
\chapter*{Foreword}

Joachim Cuntz and Danniel Quillin proposed a version of formal-smoothness, called 
\textit{quasi-freeness} on the category of associative $R$-algebras, which is much too stringent; that is almost every $R$-algebra faisl to be quasi-free \textit{(which si to say, is in that sense formally-singular)}.  The goal of non-commutative geometry is \textit{not} only to generalise our geometric setting but also to striclty gain new objects and properties without loosing any important old onces ones.  For this reason, I feel that the mass-loss of \textit{"smoothness"} by  passing to the larger category of associative $R$-algebras is an indicator of the need to modify our notions of smoothness, in order to remedy this loss thereof.  

The first goal of this short paper is to motivate and then purpose an alternatives to quasi-freeness, localization and other contructions on non-commutative algebras as an alternatice generalisations of formal smoothness, localization and more, respectivly.  

This paper is divided into two parts, first the purely algebraic results and then the algebra-geometric ones motivated from our ideas in the first part.  

In Chapter $1$ we build-up some necessary machinery, we then review some theory, and presents a result excibiting thsi mass-loss of formal smoothness, in that nearly all commutative algebras of Krull dimension atleast $2$ fail to be quasi-free.  Chapter $2$ is dedicated to presenting a new more flexible generalization of smoothnes in the non-commutative category.  Moreover, pursuing this mentality, we describe a similar approach to a nc. approximation to localization which we use to topologize the category of $R$-algebras with, in a way similar to the idea of the Zariski Site.  This leads into the second theme of the paper, which is the part on Geometry.  

In part $2$, we begin straightaways by making use of this \textit{Zariski-like} topology on the category of $R$-algebras and define nc. schemes \textit{"over a ring $R$"} in a way analogous to Demazure's Topos-based construction, this is the bulk of the work in chatper $3$.  

Chapter $4$ is essentially the heart of the paper, where we show that any nc. scheme may be approximated maximally from the inside and minimally from the outside by a scheme \textit{(in the usualy commutative sence)}.  Then we prove not only that any scheme may be embedded into our topos which we define nc. schemes in, but that essentially \textit{every} property of a morphism of schemes is preserved through this embedding, by the unique inner and outer approximations we constructed.  This motivates us to approximate essentially every possible property of a morphism of schemes to a morphism of nc. scheme, via these \textit{"inner"} and \textit{"outer"} approximations.  Since these approximatiosn are unique and are consistent with the classical theory, the present a strict and geometrically intuitive generalization of their classical schematic analogues.  We should by then be sufficiently motivated to take these to be our definitions of the nc. analogues to their classical counterparts.  This motivation is the central goal of this text.  

The very brief final par of this exposition, aims to understandand distill the properties at work in constructing any "scheme-like" object over an "appropriate" category, purely out of philosophical interest.  It is by no means meant to provide anything beyond a deeper mathematical understanding of what a scheme, group scheme, nc. scheme or any other scheme-like object is.

The appendix is then dedicated, to some new interesting geometric constructions combining nc. schemes and schemes, creating a space by amalgamating the two types of objects.  Moreover, it contains interesting abstractions of the first part, again purely for philosphical gain.  

\textit{As a final note, the reader is refered to the wonderful \textbf{nlab} and \textbf{stacks project} for help with any classical theory not reviewed in the preliminary parts of the text.  All other result are mostly, completly self contained. }

\pdfbookmark[1]{\contentsname}{tableofcontents}

\setcounter{tocdepth}{2} 
\setcounter{secnumdepth}{3} 

\tableofcontents

\newtheorem{pdefn}{Prototypical Definition}
\newtheorem{defn}{Definition}
\newtheorem{lem}{Lemma}
\newtheorem{prop}{Proposition}
\newtheorem{cor}{Corollary}
\newtheorem{thrm}{Theorem}
\newtheorem{dthrm}{Dual Theorem}
\newtheorem{esax}{Eilenberg–Steenrod Axiom}
\newtheorem{dcax}{$\mathfrak{D}$-$\mathfrak{C}$-Axiom}

\newtheorem{ex}{Example}



\part{Algebra}
\chapter{Preliminaries}
\textbf{\textit{Let us fix some notation.  }}

\textit{From herein end, $R$ is a non-associative ring, $\mathfrak{_RAlg}$ is the category of $R$-algebras and homomorphisms of $R$-algebras, $\mathfrak{_RAlg}^1$ is the full subcategory of $\mathfrak{_RAlg}$ who's $R$-algebras are unital and $_RAlg$ is its subcategory of commutative $R$-algebras an unit-preserving $R$-algebra homomorphisms.  }

\textit{Finally, $[\mathfrak{_RAlg}:Sets]$, will denote the functor category, with objects the functors from $\mathfrak{_RAlg}$ to $Sets$ and morphisms natural transformations of functors.  }

\section{Building Functorial Tools}
Before stating and formalizing one of the two main concepts of this section we will need more functors.  

The first tool we build will help push all our algebras to a category of algebras over a commutative base ring. 
\begin{prop}
There is a faithful functor $\tilde{i}:\mathfrak{_RAlg}\rightarrow \mathfrak{_Z(R)Alg}$.
\end{prop}
\begin{proof}
There is an $R$-monomorphism $i:Z(R)\rightarrow R$, corresponding to the inclusion of the center of the ring $R$, into itself.  Therefore, if $A$ is an $R$-module then sending the $R$-module homomorphism $f:R\times A \rightarrow A$ to its prescomposition with the $R$-algebramorphism $i \times 1_A$ defines the structure of an $Z(R)$-module on $A$.  Now, $R$-module homomorphisms may be viewed as $Z(R)$-module homomorphisms.  

Particularly, if $A$ is an $R$-algebra, then $A$ may be viewed as a pair $<M,f>$ of an $R$-module $M$ and $R$-module homomorphism $f: M\times M \rightarrow M$.  Therefore, since $M$ may be given the structure of a $Z(R)$-module via $i$ and $f$ may be viewed as a $Z(R)$-module homomorphism from, $M\times M$ to $M$, then $i$ induces a functor from $\tilde{i}: _R\mathfrak{Alg} $ to $_R\mathfrak{_RAlg}$.  

By construction $\tilde{i}$ is faithful.  
\end{proof}

\begin{defn}
\textbf{Unitization}

Let $A$ be an $R$-algebra, then the unitization $A_1$ of $A$ is the $R$-algebra defined as: $A \oplus R$, with operations defined as (for $a,b \in A$ and $r,s \in R)$:

$(a,r)+(b,s):= (a+b,r+s)$

$r(a,s):=(ra,rs)$, where $ra$ is the left action of $R$ on $A$ and $rs$ is the ring multiplication in $R$.

$(a,s)r:=(ar,sr)$, where $ar$ is the right action of $R$ on $A$ and $rs$ is the ring multiplication in $R$. 

$(a,r)(b,s):=(ab+rb+sa,rs)$, where $ab$ is the multiplication in $A$, $rs$ is the multiplication in $R$ and $rb$ as well as $sa$ are left actions of $R$ on $A$.  

$(a,r)(b,s):=(ab+rb+sa,rs)$

And multiplicative unit:

$(0_A,1_R)$.  
\end{defn}
The verification that this indeed forms an algebra is a straightforward verification from the definition.  
\begin{prop}
The unitzation defines a faithful functor from $\mathfrak{_RAlg}$ to $\mathfrak{_RAlg}^1$.
\end{prop}
\begin{proof}
Any morphism $f:A \rightarrow B$ of $R$-algebras is taken to the morphism $(f,1_R)$, the restrictions simply follow from the properties of the restrictions on the morphism $f$, since $1_R$ is a morphism of commutative unital $R$-algebras (by assumption on $R$ being commutative and unital).  
\end{proof}

\begin{defn}
\textbf{The abelianisation of an algebra}

For any $R$-algebra $A$, the $R$-algebra $A^{ab}$ defined as the quotient of $A$, by its commutator subgroup $[A,A]$ is called the \textbf{abelianisation of A}.  
\end{defn}
\begin{prop}
The abelianisation of an algebra is defines a faithful endofunctor on $\mathfrak{_RAlg}$, with $A^{ab}$ being a commutative algebra.  

Moreover if, $-^{ab}$ restricts to an faithful functor from on $\mathfrak{_RAlg}^1$ to $_RAlg$.   
\end{prop}

\begin{proof}
$[A,A]$ is an ideal in $A$, therefore the quotient $A^{ab}$ is again an algebra, moreover by the universal property of the quotient of an algebra, for every $R$-algebra homomorphism $f: A \rightarrow B$, there is a unique $R$-algebra homomorphism $\tilde{f}:A^{ab} \rightarrow B$ descending to $A^{ab}$, such that $\pi \circ \tilde{f}=f$, where $\pi$ is the canonical epimorphism $A \rightarrow [A,A]$.  Hence, $-^{ab}$ is a functorial endofunctor.  Post-composing with the canonical map $\tilde{pi}:B \rightarrow B^{ab}$ gives a unique morphism from $f^{ab}:A^{ab} \rightarrow B^{ab}$.  

Let $a,b \in A^{ab}$ then $ab= ab+[A,A] = ab +(-ab+ba) =ba$, hence $A^{ab}$ is a commutative $R$-algebra.  

If $A$ is commutative then trivially $A^{ab} \cong A$, moreover if $A$ is unital $1+[A,A]$ is the multiplicative unit in $A^{ab}$; finally if an $R$-algebra homomorphism $f:A\rightarrow B$ preserves units then, $f^{ab}(1_A+[A,A])\mapsto f(1_A)+[B,B]=1_B+[B,B]$ hence, so does the induced map $f^{ab}$.  

The construction implicitly showed that $-^{ab}$ is full, by definition of $A^{ab}$ as a quotient algebra.  
\end{proof}

\subsubsection{Our most important functor}
The composition of faithful functors is again faithful functor, therefore at once we have:
\begin{cor}
$-^{ab} \circ -^1\circ \tilde{i}$ defines a faithful functor from the category $\mathfrak{_RAlg}$ to the full category of $_Z(R)Alg$ with objects the unital commutative $R$-algebras.
\end{cor}

Since we will only, and very much be interested in this functor lets give it a name:
\begin{defn}
\textbf{Standardization}

The functor $-^{ab} \circ -^1\circ \tilde{i}$ described in the above proposition, is called the \textbf{Standardization} functor, and is denoted $-^{\varsigma}$.  Moreover, for any $R$-algebra $A$, $A^{\varsigma}$ will be called its \textbf{Standardization}.  
\end{defn}
\break

\section{Degeneracy of Smoothness}

We begin by recalling this very important and widely known characterisation of smoothness, who's proof may be found on the stacks project, under the commutative algebra section.  

In the case where the basefield is algebraically closed, for example in the case where it is $\mathbb{C}$, the following result is a classical characterisation of smoothness of a scheme: 

\begin{thrm} $\label{thrm:smoothalgclo}$
Let X be a scheme of fintie type over an algebraically closed field k and $\{ Spec(A_i) \} $ an open cover of X. Then the following are equivalent:

1)  X is a smooth scheme

2)  Each $A_i$ is a regular ring 

3)  $\Omega_{X|Spec(k)}$ locally free

4)  For every affine open cover $Spec(A_i)$ of X, each $\Omega_{A_i|k}$ is a locally free $A_i$-module 

5)  For every affine open cover $Spec(A_i)$ of X, each $\Omega_{A_i|k}$ is a projective $A_i$-module 

\end{thrm}

For simplicity we from now on, assume $k$ to be a field and $A$ to be an Noetherian Regular $k$-algebra.  
The goal of this section will be to prove the following theorem, which will be undergone in a series of lemmas.

\begin{thrm} $\label{prop:wkDC}$
\textbf{Weak Degeneracy Criterion}

\textbf{0)}  If $A$ is a commutative unital $k$-algebra over a commutative ring $k$, then $A$'s Hochschild cohomological dimension $HCHDim(A)$ is bounded below by $wd(A)$ and in particular by $pd(A)$.  

Moreover if:

\textbf{1)  $A$ is Cohen-Macaulay at some maximal ideal $\mathfrak{m}$}

Then $Krull(A_{\mathfrak{m}})\leq HCHDim(A)$.  \textit{In this scenario: if $A_{\mathfrak{m}}$ is of Krull dimension at least $2$ then $A$ is not Quasi-free}.  

\textbf{2)  $A$ is Regular}

Then $Krull(A) \leq HCHDim(A)$.  
\textit{In this scenario: if $A$ itself is of Krull dimension at least $2$ then $A$ is not Quasi-free}.  

\end{thrm}

The terminology will be explained as the section progresses, but in a nutshell the idea of the theorem and the proof is that if a $k$-algebra $A$'s abelianization with a unit adjoined has a maximal ideal $\mathscr{m}$ such that the localization thereat is of dimension atleast 2.  Then the Hochschild cohomological dimension of $A$ may be bounded below by $2$, Cuntz's and Qullin's characterisation of Quasi-freeness implies that $A$ cannot be quasi-free; that is $A$ cannot be \textit{"smooth"} as a non-commutative algebra.  

\subsection{Some Algebraic Background}
\subsubsection{Some Dimension Theory}
We will need two notions, the notion of the \textit{projective dimension} and that of the \textit{flat dimension} of an $A$-bimodule, and theore of an algebra.  

\begin{defn} $\label{defn:fd}$
\textbf{Flat Dimension}

The flat dimension $fd(M)$ of an $A$-bimodule is defined the be the extended natural number $n$, defined as the shortest \textit{length} of a deleted resolution of $M$ by flat $A$-bimodules.  
\end{defn}

\begin{ex}$\label{ex:fd}$
One of the simplest examples of flat dimension, is of a projective $A$-bimodule $M$, where $fd(M)=0$.  
\end{ex}
\begin{proof}
If $M$ is $A$-flat, then $M$ admits the a projective resolution: $0 \rightarrow M \rightarrow P_0 \rightarrow P_1 \rightarrow ...$.  
Where the projective bimodule $P_0$ is $M$ and for $i$ greater than $0$ are, the other projective bimodules $P_i$ are the trivial bimodule $0$.  

This deleted projective resolution is of length $0$.  Since projective modules are flat this is a flat resolution of length $0$, since the flat dimension cannot be negative we have $fd(M)=0$ if $M$ is projective.  
\end{proof}

We will need the following characterisation of flat dimension:
\begin{lem} $\label{char:fd}$
The following are equivalent:

-The flat dimension of the $A$-bimodule $M$ almost $n$.

-For every left $A$-module $N$, $Tor_{n+1}^A(M,N)$ is the trivial group.  

\end{lem}
\begin{proof}

\end{proof}

The notion of \textit{flat dimension} will be used to provide an explicit lower bound for the \textit{projective dimension}.  As one would expect the projective dimension of an $A$-bimodule is defined as:

\begin{defn} $\label{defn:pd}$
\textbf{Projective Dimension}

The projective dimension $fd(M)$ of an $A$-bimodule is defined the be the extended natural number $n$, defined as the shortest \textit{length} of a deleted projective resolution of $M$.  
\end{defn}
An immediate consequence of this definition and generalisation of the previous example is the following lemma:
\begin{lem} $\label{lem:fdpd}$

For any $A$-bimodule $M$, its flat dimension is almost its projective dimension.  
\end{lem}
\begin{proof}
Since all flat $A$-modules are projective, then any flat resolution is a projective resolution.  
\end{proof}

We now provide a similar characterisation of projective dimension as we did for the flat dimension, that will play a central role in our proof.  
\begin{lem} $\label{char:pd}$

The following are equivalent:

-The projective dimension of the $A$-bimodule $M$ almost $n$.

-For every left $A$-module $N$, the group $Ext_{n+1}^A(M,N)$ is trivial.  

\end{lem}
\begin{proof}

\end{proof}

Before being able to utilise this powerful machinery we need some commutative algebraic notions.  
\begin{defn}
\textbf{Regular element}

A non-zero element $x$ in an $A$, is said is said to be regular on $M$, or $M$-regular if the the $A$-module endomorphism $m \mapsto xm$ where $xm$ is the action of $x$ on $m$ is an injection.  
\end{defn}

\begin{ex}$\label{ex:ringreg}$
If $R$ is a ring of characteristic $0$, then the element $x$ in the polynomial ring $R[x,y]$ is regular.  
\end{ex}

Our use for regular elements is that they provide a constructive way of building up a sharp lower bound of the flat dimension of an $A$-bimodule.  Combining this with the lemma $\autoref{lem:fdpd}$ gives a concrete lower bound for the projective dimension of $M$.  
\begin{lem} $\label{lem:Regular}$

If $M$ is an $A$-bimodule and $x$ is $M$-regular, then $fd(M)=fd(M/(x)M+1$.  
\end{lem}
\begin{proof}
Widely known, see stacks project.  

\end{proof}

\begin{ex} $\label{ex:zdim}$
The flat dimension of $\mathbb{Z}$ as a $\mathbb{Z}[X]$-bimodule, is exactly $n$ and its projective dimension if atleast $n$.  
\end{ex}
\begin{proof}
$\mathbb{Z}[X]$ is free, and theore projective as a $\mathbb{Z}[X]$-bimodule.  Theore, in example $\autoref{ex:fd}$ it was remarked that the projective dimension of $\mathbb{Z}[X]$ must theore be equal to $0$.  

Example $\autoref{ex:ringreg}$ showed that the polynomial $X$ must be $\mathbb{Z}[X]$-regular, and so $fd(\mathbb{Z})=1$ as a $\mathbb{Z}[X]$-bimodule.  
Lemma $\autoref{lem:fdpd}$ then implies that $1 \leq pd(\mathbb{Z})$ as $\mathbb{Z}-bimodule$.  
\end{proof}

We iterate the above lemma to generate a more rapidly applicable one.  
\begin{defn}
\textbf{M-regular sequence}

A sequence $x_1,..,x_n$ of elements of $A$ are said to form an \textbf{$M$-regular sequence} if and only if each $x_{i}$ is regular on $M/(x_0,..,x_{i-1})M$, where $x_{0}$ is taken to be $0$.  

The regular sequence $x_1,..,x_n$ is said to be of \textbf{maximal length} if for every other $M$-regular sequence $y_1,..,y_m$ in $A$ $m\leq n$.  
\end{defn}
\begin{lem} $\label{lem:fd}$

If $x_1,..,x_n$ is an $M$-regular sequence on the $A$-bimodule $M$, then 

$fd(M/(x_1,..,x_{n})M)=-n+fd(M)$.  
\end{lem}
\begin{proof}
The definition of an $M$-regular sequence implies that, $x_1$ is regular on $M/(x_1)M$ and so the last implies $fd(M)-1 = fd(M/(x_1)M)$.  Likewise for every $1 < i\leq n$ the definition of an M-regular sequence implies that $x_i$ is regular on $M/(x_{i-1})M$ and so lemma $\autoref{lem:Regular}$ implies that $fd(M/(x_1,..,x_{i-1})M)= fd(M/(x_1,..,x_{i})M) -1 $.  Iteration gives $fd(M)= fd(M/(x_1,..,x_{n})M) +n$.  
\end{proof}
\begin{ex} $\label{ex:zxdim}$
Example $\autoref{ex:zdim}$ and lemma $\autoref{lem:fd}$ then straight-away imply that the polynomial ring $\mathbb{Z}$ is of flat dimension $n$ as a $\mathbb{Z}[x_1,..,x_n]$-bimodule.  
\end{ex}

The types of algebras we will be most interested are those which are \textit{Cohen-Macaulay} when localised at a certain ideal.  
\begin{defn} $\label{def:CM}$
\textbf{Cohen-Macauley at an ideal}

A commutative unital $R$-algebra $A$ is said to be \textbf{Cohen-Macaulay} at an ideal $I$, if there is an $A_I$-regular sequence $x_1,...,x_d$ of maximal length $d$, in $A_I$ where $d$ is the Krull dimension of $A_I$.  
\end{defn}

\begin{ex} $\label{ex:ZXCM}$
$\mathbb{Z}[x_1,..,x_n]$ is Cohen Macaulay at the maximal ideal $(x_1-z_1,..,x_n-z_n,p)$.  
\end{ex}
\begin{proof}
The ring $\mathbb{Z}[x_1,..,x_n]$ is of Krull dimension $n+Dim(\mathbb{Z})=n+1$ and $(x_1-z_1,..,x_n-z_n,p)$, where $p$ is a prime number and $z_1,..,z_n$ are arbitrary integers is a maximal ideal therein .  This implies that $p,x_1,..,x_n$ is a regular sequence in $\mathbb{Z}[x_1,..,x_n]$ ; moreover it must be maximal since $(p)\lhd...\lhd (p,x_1-z_1,..,x_n-z_n)$ is a proper chain of non-trivial ideals of height $n+1$.  

The correspondence between localized prime ideals and of primeideals, then implies that 
$\frac{p}{1},\frac{x_1-z_1}{1},..,\frac{x_n-z_n}{1}$ is a regular sequence of maximal length in the local ring 

$\mathbb{Z}[x_1,..,x_n]_(p,x_1-z_1,..,x_n-z_n)$.  

In other words, 
\end{proof}

The central use of this notion will be to localize a commutative algebra $A$ at a certain maximal ideal $\mathfrak{m}$, at which it is Cohen Macaulay; a regular sequence may then be obtained of length equal to the krill dimension of $A_{\mathfrak{m}}$ which will then be used to provide a lower bound for the projective dimension of $A$.  

\begin{prop} $\label{prop:Krullpd}$
If $A$ is Cohen Macaulay at the maximal ideal $\mathfrak{m}$, then $Krull(A_{\mathfrak{m}})\leq pd(A)$.  
\end{prop}
\begin{proof}
If $A$ is Cohen Macaulay at $\mathfrak{m}$, there is a regualr sequence $x_1,..,x_d$ of length $d:=Krull(A_{\mathfrak{m}})$ in $A_{\mathfrak{m}}$.  Applying lemmma $\autoref{lem:fd}$ implies 

$fd(A_{\mathfrak{m}}/(x_1,..,x_d)A_{\mathfrak{m}}) = -d + fd(A_{\mathfrak{m}})$.  
Hence, $fd(A_{\mathfrak{m}})\geq d = Krull(A_{\mathfrak{m}}).  $

Now applying lemma $\autoref{lem:fdpd}$ implies $pd (A_{\mathfrak{m}}) \geq fd(A_{\mathfrak{m}}) \geq Krull(A_{\mathfrak{m}})$.  Finally, since the restriction of scalars functor induced through the localization map $A \rightarrow A_{\mathfrak{m}}$ implies that every $A_{\mathfrak{m}}$-module is an $A$-module, and so $pd(A_{\mathfrak{m}}) \leq pd(A)$.  

\end{proof}

There are a variety of characterisations of regularity, one that will now provide an immediate use for our interests in the following, proven in Eisenbud .  

\begin{lem} $\label{lem:Eisenbud}$
A ring $R$ is regular of krull dimension $d$ if and only if for each maximal ideal $\mathfrak{m}$ there is a maximal $R_{\mathfrak{m}}$-regular sequence $x_1,..,x_d$ of length $d$, moreover $d = Krull (R_{\mathfrak{m}})$.  
\end{lem}
\begin{proof}
Known, see stacks project, commutative algebra section.  
\end{proof}

This allows for an immediate particularisation of example $\autoref{ex:ZXCM}$ as follows:
\begin{prop} $\label{propr:regCM}$
A commutative unital associative $R$-algebra $A$ is Cohen-Macaulay at each maximal ideal if it is regular.  
\end{prop}
\begin{proof}
If $\mathfrak{m}$ is a maximal ideal in $A$, then $A_{\mathfrak{m}}$ is of krull dimension $Krull(A_{\mathfrak{m}})$, since $A$ is regular; which is the definition of being Cohen-Macaulay at each maximal ideal.  
\end{proof}

A straightforward geometric example, 

with a classical geometric flavour is the following.

\begin{ex}
If $X$ is a smooth affine scheme over $Spec(\mathbb{Z}/p\mathbb{Z})$, then its coordinate ring $\mathbb{Z}/p\mathbb{Z}[X]$ is Cohen-Macaulay at each closed point.  
\end{ex}
\begin{proof}
Since $\mathbb{Z}/p\mathbb{Z}$ is an algebraically closed field  then theorem $\autoref{thrm:smoothalgclo}$ implies that $\mathbb{Z}/p\mathbb{Z}[X]$ is regular.  Proposition $\autoref{propr:regCM}$ above, now implies that $\mathbb{Z}/p\mathbb{Z}[X]$ is Cohen-Macaulay at each maximal ideal in $\mathbb{Z}/p\mathbb{Z}[X]$.  By definition of the Zarisky topology on an affine scheme $(: DF)$ there are exactly the closed points of $Spec(\mathbb{Z}/p\mathbb{Z}[X])$.  
\end{proof}

The above proposition particularises in the case that the algebra in question is regular as:
\begin{cor} $\label{prop:Krullpdreg}$
If $A$ is a regular $R$-algebra, then $Krull(A)\leq pd(A)$.  
\end{cor}
\begin{proof}
Pick some maximal ideal $\mathfrak{m}$.  Since $A$ is regular, then lemma $\autoref{propr:regCM}$ implies it is Cohen-Macaulay at $\mathscr{m}$ and proposition $\autoref{prop:Krullpd}$ implies that $A_{\mathfrak{m}}\leq pd(A)$.  However since $A$ is regular then lemma $\label{lem:Eisenbud}$ implies that $Krull(A_{\mathfrak{m}})=Krull(A)$ and so $Krull(A) = Krull(A_{\mathfrak{m}}) \leq pd(A)$.  
\end{proof}
\subsection{The weak degeneracy criterion}
We are now essentially in place to prove a the degeneracy of smoothness criterion, however we'll need one last lemma:

\begin{lem} $\label{lem:pre-weakWeibeliso}$
Let $R$ be a commutative ring, and $A$ be an $R$-algebra and $M,N$ be left $R$ then there are natural isomorphisms:

$Ext_{A/k}^n(M,N) \cong Ext_{A^e}^n(A,Hom_k(M,N))\cong HH^n(A,Hom_k(M,N))$,

where $Ext_{A/k}^n(-,-)$ is the $k$-relative $ext$-bifunctors.  
\end{lem}
\begin{proof}
The first remark is that the left $A$-module $Hom_k(M,N)$ has the structure of a $A$-bimodule via the action $(a,a')f(-) \mapsto af(a'-)$, where $a$ and $a'$ are elements of $A$ and $f$ is a $k$-module homomorphism from M to N; $af$ is the inherit multiplication in $Hom_k(M,N)$ and $f(a'-)$ is the evaluation of $f$ is the pre composition of $f$ with the left-multiplication endomorpihsm $a'-$ on $M$.  Since $(af)a' = (af)(a'-) = af(a')=a(fa')$ then indeed the left and right structure as defined on $Hom_k(M,N)$ are compatible, ensuring that indeed $Hom_k(M,N)$ is an $A$-bimodule.  

The categories $_{A^e}Mod$ and $_AMod_A$ are equivalent, theore $Hom_k(M,N)$ is indeed an $A^e$-bimodule and the expression $Ext_{A^e}^n(A,Hom_k(M,N))$ is meaningful.  

Now the natural isomorphism $Ext_{A^e}^n(A,Hom_k(M,N))$

$\cong HH^n(A,Hom_k(M,N))$ was demonstrated earlier .  

We theore only need still show the first isomorphism: $Ext_A^n(M,N) \cong Ext_{A^e}^n(A,Hom_k(M,N))$.  

The Hochsild cohomology of a \textit{unital} algebra $A$ with coefficients in $Hom_k(M,N)$ is defined as the homology of the complex 

$Hom_{A^e}(C,Hom_k(M,N))$.  Since $A$ is a $k$-algebra then any $A$-module can be viewed as a $k$-module, theore the product $_A\otimes_k M$ may be in-turn be seen as a functor from $_kMod$ to $_AMod$.  The tensor-hom adjunction them implies, 

$_A\otimes_k$ has as right adjoint the functor $Hom_k(M,-):_AMod \rightarrow _kMod$.  In other words, there are natural isomorphisms $Hom_A(A\otimes_k M, N) \cong Hom_{_AMod_A}(A,Hom_k(M,N))$.  This implies the first isomorpihsm.  

Now let $CH^{\star}(A)$ be the Hochschild cocomplex of $A$, then the above discussion implies there are natural isomorphisms $Hom_A(CH^{\star}(A)\otimes_k M, N) \cong$ 

$Hom_{_AMod_A}(CH^{\star}(A),Hom_k(M,N))$.   In weibel it is seen that, there are isomrophisms $Hom_{_AMod_A}(CH^{\star}(A),Hom_k(M,N))$ 

$\cong HH^{\star}(A,Hom_k(M,N))$.  

Since the bar resolution may be used to compute 

the homology groups of an algebra then we have

$Hom_A(A\otimes_k M, N)$ 

$\cong Hom_{_AMod_A}(A,Hom_k(M,N))$ 

$\cong HH^{\star}(A,Hom_k(M,N))$.  
\end{proof}

This small technical lemma is of total importance in our of effort to bound the Hochschild cohomological dimension of our algebra below, by the Krull dimension of one of one its localizations.  The first evident major step with our goals is illustrated in the proof of a weak version of the Degeneracy criterion:

\begin{defn} $\label{defn:wDim}$
\textbf{Weak Dimension}

The \textbf{weak dimension} $wD(M)$ of an $R$-bimodule $M$, is defined as the supremum of the flat dimensions of its $R$-bimodules.  That is:

$wD(R):=\underset{M\in _AMod_A}{sup} pd(M)$.  
\end{defn}

\begin{thrm} $\label{prop:wkDC}$
\textbf{Weak Degeneracy Criterion}

\textbf{0)}  If $A$ is a commutative unital $k$-algebra over a commutative ring $k$, then $A$'s Hochschild cohomological dimension $HCHDim(A)$ is bounded below by $wd(A)$ and in particular by $pd(A)$.  

Moreover if:

\textbf{1)  $A$ is Cohen-Macaulay at some maximal ideal $\mathfrak{m}$}

Then $Krull(A_{\mathfrak{m}})\leq HCHDim(A)$.  \textit{In this scenario: if $A_{\mathfrak{m}}$ is of Krull dimension at least $2$ then $A$ is not Quasi-free}.  

\textbf{2)  $A$ is Regular}

Then $Krull(A) \leq HCHDim(A)$.  
\textit{In this scenario: if $A$ itself is of Krull dimension at least $2$ then $A$ is not Quasi-free}.  

\end{thrm}
\begin{proof}
For any $A$-bimodule $M$, we have $Ext_{A/k}^{\star}(A,M)$ 

$\cong HH^{\star}(A,Hom_k(A,M))$, therefore in particualr the supremum the $R$-projective dimenions of all $A$-module $M$ is bounded above by: $wD(A):=\underset{M\in _AMod_A}{sup}$ 

$HH^{\star}(A,Hom_k(A,M)) \leq \underset{N\in _AMod_A}{sup} HH^{\star}(A,N)$.  Proving claim $0$.  

Since the weak dimensoin $wD(A)$ is defined as the supresmum of all the projective dimension of \textit{every} A-bimodule, then it is bounded below by the projective dimensiona of any $A$-biomodule.  In particular so for the $A$bimodule, $A_{\mathfrak{m}}$.  IF $A$ is Cohen-Macaulay at some maximal ideal $\mathfrak{m}$, proposition $\autoref{prop:Krullpd}$ therefore implies gives a lowerbound for $wD(A)$ is $Krull(A)$.  Tying this in with the above observed fact that $wD(A)\leq HCHDim(A)$ implies $Krull(A)\leq HCHDim(A)$, showing the claim $1$.  

Nowm proposition $\autoref{propr:regCM}$ implies the finaly claim $2$.  
\end{proof}

A very particular example of the above theory is:
\begin{cor}
If $k$ is an algebraically closed field, then if $A$ is smooth of dimension at least $2$ then $A$ is both smooth and not quasi-free.  
\end{cor}
\begin{proof}
The characterisation of smoothness given in theorem $\autoref{thrm:smoothalgclo}$, implies that over a field an algebra $A$ is smooth \textit{if and only if} it is regular.  Combining this with proposition $\autoref{prop:wkDC}$-2 implies $A$ must fail to be quasi-free also.  
\end{proof}
This last result is quite cool, since it shows that smooth algebras no longer correspond to \textit{"smooth"} non-commutative geometric objects, once considered as in the category of associative algebras.  

This version of the result is sufficient for observations of loss of smoothness of commutative algebras once transported to the category of associative algebras.  However, I per develop a tool, that applies to any old associative algebra over any old \textit{unital} ring, which essentially makes use of this idea.

\subsection{A Host of examples}
Now we may reap the rewards of our efforts.  Lets start with a big class of algebras which degenerate.  

\begin{ex}
If $k$ is an algebraically closed field and $G$ is a smooth affine group scheme of topological dimension atleast $2$, then $k[G]$ fails to be quasi-free.  
\end{ex}
\begin{proof}
Since $G$ is an affine group scheme then, in particular it is an affine scheme.  The topological dimension of an affine scheme is exactly equal to the Krull dimension of its coordinate ring.  Since the base-field is algebraically closed, the smoothness of $G$ is equivalent to the regularity of $k[G]$ by theorem $\autoref{thrm:smoothalgclo}$.  Theore, $Krull(k[G])\leq 2$ and theore we are in scenario $2$ of the weak degeneracy criterion (theorem $\autoref{prop:wkDC}$).  Hence, $k[G]$ cannot be quasi-free.  
\end{proof}

Lets give a more concrete example:
\begin{ex}
Let $C_n$ be a complex $n$-sphere, that is $C_n$ is the affine scheme 

$Spec(\mathbb{C}[X_1,..,X_n]/(\underset{i=1}{\overset{n}{\sum}} X_i^2 -1))$,

then if $n$ is greater than $2$ the corresponding canonical \textit{"non-commutative space"} to the $\mathbb{C}$-algebra $\mathbb{C}[X_1,..,X_n]/(\underset{i=1}{\overset{n}{\sum}} X_i^2 -1)$ is not quasi-free.  
\end{ex}

\subsection{Lie algebraic Examples}
We now discuss the quasi-freeness of a lie algebra $\mathfrak{g}$ over a field $k$.  


The Jacobi identity implies that in general, a lie algebra need not be associative.  This problem is remedied with the construction of the \textit{universal enveloping algebra} $U(g)$ of a lie algebra.  
\begin{ex} $\label{ex:End}$
The endomorphism ring $End_k(M)$ of a $k$-module has the structure of a Lie algebra, with Lie bracket, its commutator.  
\end{ex}
\begin{proof}
The endomorphism ring of a $k$-module has natural action on $M$, taking a pair $f \otimes m$ of an endomorphisms $f$ and an element $m$ in $M$ and mapping them to the element $f(m)$ of $M$.  Since $1_M\cdot m=1_M(m)=m$ and $(fg)\cdot m = (fg)(m)=f\circ g(m) = f(g\cdot m) = f\cdot (g \cdot m)$ for every $f,g \in End_k(M)$ and $m \in M$ then this in fact does define an action of $End_k(M)$ on $M$.  

\end{proof}

\subsubsection{A brief recap of lie algebra cohomology}

For any lie algebra $\mathfrak{g}$, the universal enveloping algebra is a pair $<U(g),\iota >$ of a \textit{unital} associative algebra $U(g)$ together with a map $\iota: \mathfrak{g} \rightarrow U(\mathfrak{g})$, such that for any assoiciative algebra A and any \textit{linear} map $\pi :\mathfrak{g} \rightarrow U(\mathfrak{g})$, respecting the lie bracket; $\pi$ factors uniquely through $U(\mathfrak{g})$.

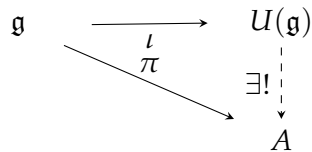
\begin{tikzpicture}
  \matrix (m) [matrix of math nodes,row sep=2.5em,column sep=4em,minimum width=5em]
  {
     \mathfrak{g} & U(\mathfrak{g}) \\
       & A \\};
  \path[-stealth]
    (m-1-1) edge node [above] {$ \pi $} (m-2-2)
    (m-1-1) edge node [below] {$ \iota $} (m-1-2)
    (m-1-2) edge [dashed] node [left] {$ \exists ! $} (m-2-2);
\end{tikzpicture}

Since $U(\mathfrak{g})$ is an associative algebra then, left $U(\mathfrak{g})$-modules are definable in the usual way.  An equivalent category $_{\mathfrak{g}}Mod$, of left $\mathfrak{g}$-modules is defined in a slightly different manner.  

\begin{ex} $\label{ex:ug}$
For any module $M$, there are exactly as many $k$-algebra homomorphisms from $U(\mathfrak{g})$ to $End_k(M)$ as there are Lie-algebra homomorphisms from $\mathfrak{g}$ to $End_k(M)$.  
\end{ex}
\begin{proof}
Let $M$ be a $k$-module.  

In example $\autoref{ex:ug}$ it was observed that $End_k(M)$ had the structure of a Lie algebra.  Theore if $\lambda: \mathfrak{g} \rightarrow End_k(M)$ is a lei-homomorphism then the universal property of the universal enveloping algebra implies that there is a \textit{unique} $k$-algebra homomorphism from $U(\mathfrak{g})$ to $End_k(M)$ agreeing with $\mathfrak{g}$ both as viewed within $U(\mathfrak{g})$ and not.  

Conversely if $\phi: U(\mathfrak{g}) \rightarrow End_k(M)$ is a homomorphism of $k$-algebras.  Then, since $\iota: \mathfrak{g} \rightarrow U(\mathfrak{g})$ is a $k$-homomorphism then in particular so must $\phi \circ \iota$ be.  However, by definition $\iota$ 

respects the lie bracket of $\mathfrak{g}$, theore for every $g,h \in \mathfrak{g}$ $\phi \circ \iota ([g,h]) = \phi (\iota (g) \iota (h) -\iota(h)\iota(g) )$, since $\phi$ is a $k$-algebra 

homomorphism then this expression equals $\phi (\iota (g)) \phi (\iota (h)) -(\phi \iota(h))\phi (\iota(g))$ 

$= (\phi \circ \iota (g), \phi \circ \iota (h) )$.  Since, the commutator bracket $(-,-)$ 

is the lie bracket of $End_k(M)$ then precomposition with the $k$-algebra homomorphism $\iota$ induces a Lie-homomorphism $\phi \circ \iota: \mathfrak{g} \rightarrow End_k(M)$.  

Verifying that these associations are indeed inverses of each-other implies the result.  
\end{proof}

\subsubsection{left $\mathfrak{g}$-modules}
Viewing $\mathfrak{g}$ simply as a $k$-algebra, if there where to be a $k$-algebra homomorphism $\lambda: \mathfrak{g} \rightarrow End_k(M)$ to the endomorphism ring of $M$, then there would be a canonical way to view $M$ as a left $\mathfrak{g}$-module via the action $g \otimes m \mapsto \lambda(g)\cdot m$.  However, this would only transfer the underlying $k$-algebra structure of $\mathfrak{g}$ to $M$ and is does not represent its lie algebraic properties.  

The idea is there, if we impose more structure on the $k$-algebra homomorphism $\lambda$ we can accomplish this.  Since, the commutator on a $k$-module satisfies all three of the conditions for a lie bracket (: Stammback and Hilton), the endomorphism ring $End(M)$ with multiplication given by its commutator then defines a lie algebra, which we denote again $End_k(M)$.  With this in mind, if $\lambda: \mathfrak{g} \rightarrow End_k(M)$ is a lie-algebra homomorphism then we would have that the left $\mathfrak{g}$-module structure on $M$ now has the property that, for any $g,h \in \mathfrak{g}$ and any element $m \in M$, $[g,h]\cdot m$ which is by definition equal to $\lambda([g,h])\cdot m$ must be equal to $(\lambda(g),\lambda(h))\cdot m$, where $(-,-)$ is $End_k(M)$'s commutator bracket.  
Expanding gives $(\lambda(g)\lambda(h))\cdot m - (\lambda(h)\lambda(g))\cdot m$, since this is just the action of $End_k(M)$ on $M$ it must equate to $\lambda(g)(\lambda(h)\cdot m) - \lambda(h)(\lambda(g)\cdot m)$.  Which is representative of $\mathfrak{g}$'s lie algebra structure.  

We summarise that all that was necessary in the above definition was a lie-algebra homomorphism $\lambda$ from $\mathfrak{g}$ to the endomorphism ring $End_k(M)$ of a left $k$-module to be able to define a left $\mathfrak{g}$-module which had \textit{"lie- algebraic-like structure"}.  These special left $\mathfrak{g}$-modules (where $\mathfrak{g}$ is considered only as a $k$-algebra) are:

\begin{defn} $\label{defn: gmod}$
\textbf{left $\mathfrak{g}$-module}

A left $\mathfrak{g}$-module is a pair $M:= <M',\lambda >$, of a $k$-module $M'$ and a lie-algebra homomorphism $\lambda: \mathfrak{g} \rightarrow End_k(M)$.  
\end{defn}

\begin{ex}
In particular, if $M$ and $N$ are left $\mathfrak{g}$-modules, then $Hom_k(M,N)$ can be made into a left $\mathfrak{g}$-module, with 

action by acting first and then after, by an element, $g$ of $\mathfrak{g}$ and taking the difference.  That is: $g\otimes f := g\cdot f - f(g\cdot )$.  
\end{ex}
Another interesting example is:
\begin{ex}
Any $k$-module $M$ can be considered as a trivial $\mathfrak{g}$-module $M_{tr}:=<M,0>$ where $0$ is the zero lie-algebra homomorphism.  That is the action of an element $g$ of $\mathfrak{g}$ on an element $m$ in $M$, is  defined by $g \cdot m  := 0$.  
\end{ex}

\begin{defn}
\textbf{$\mathfrak{g}$-homomorphism}

If $M$ and $N$ are left $\mathfrak{g}$-modules with underlying, then a $\mathfrak{g}$-homomorphisms between $M$ and $N$ is just a $k$-algebra homomorphism from $M$ to $N$.  
\end{defn}

Since composition of $\mathfrak{g}$-homomorphism, is theore inherited from their composition laws as $k$-Modules and by definition since every $k$-module homomorphism is a $\mathfrak{g}$-homomorphism, and visa versa it follows that:
\begin{prop}
For any lie algebra $\mathfrak{g}$, the category of left $\mathfrak{g}$-modules and $\mathfrak{g}$-homomorphisms does indeed form a category and it is a full subcategory of $_kMod$ at that.  
\end{prop}

Besides the fact that the \textit{universal enveloping algebra} $U(\mathfrak{g})$ of a Lie algebra $\mathfrak{g}$, solved the problem of $\mathfrak{g}$ not being associative, it finds great use in the functorial definition of \textit{lie-algebra cohomomology} through the following result.  
\begin{prop}
There is an equivalence of categories, between the categories of left $\mathfrak{g}$-modules and left $U(\mathfrak{g})$-modules.   
\end{prop}
\begin{proof}
A left $\mathfrak{g}$-module was defined to be a pair of a $k$-module and a Lie algebra homomorphism from $\mathfrak{g}$ to the endomorphism ring of $M$.  

Likewise, if $\phi: (\mathfrak{g}) \rightarrow End_k(M)$ is a $k$-algebra homomorphism, then similarly to the construction of left $\mathfrak{g}$-modules, $\phi$ induces an action of $U(\mathfrak{g})$ on $M$ via $u\otimes m \mapsto \phi(u)\cdot m$, where $\cdot$ is the inherit action of $End_k(M)$ on $M$ and $u$ and $m$ are in $U(\mathfrak{g})$ and $M$, respectively.  So, left $U(\mathfrak{g})$-modules are exactly in bijection with the $k$-algebra homomorphisms from $U(\mathfrak{g})$ to $End_k(M)$.  

However, in example $\autoref{ex:ug}$ it was seen that for \textit{any} Lie algebra $\mathfrak{g}$,there was a bijection between the \textit{lie algebra homomorphism from $\mathfrak{g}$ to $End_k(M)$}, that is \textit{left $\mathfrak{g}$-modules} and the \textit{$k$-algebra homomorphism from $U(\mathfrak{g})$ to $End_k(M)$}, that is the \textit{left $U(\mathfrak{g})$-modules}.  
\end{proof}

With this in mind, a natural cohomology theory might pop into mind...
\begin{defn} $\label{def:Lac}$
\textbf{Lie Algebra Cohomology}

The \textbf{Lie Algebra Cohomology} $HL^{\star}(\mathfrak{g},M)$ of the Lie algebra $\mathfrak{g}$ with coefficients in the $\mathfrak{g}$-module $M$ is defined as:

$Ext_{U(\mathfrak{g})}^{\star}(k_{tr},M)$, where $k$ and $M$ are considered as $U(\mathfrak{g})$-modules, and in particualr $k$ is regarded with its trivial structure.  
\end{defn}

We now show in a slightly less conventional manner, that infact the \textit{lie algebra cohomology} is nothing more than a special case of the Hochschild cohomology.  
\begin{prop} $\label{prop:lHoch}$
If $\mathfrak{g}$ is a lie algebra and $M$ is a left $\mathfrak{g}$-module then there are natural isomorphisms between the groups, $HL^{\star}(\mathfrak{g},M)$ and $HH^{\star}(U(\mathfrak{g}),M)$, where $M$ is considered with its unique left $U(\mathfrak{g})$-module structure.  
\end{prop}
\begin{proof}
By definition $\autoref{def:Lac}$ $HL^{\star}(\mathfrak{g},M)$ is exactly $Ext_{U(\mathfrak{g})}^{\star}(k_{tr},M)$.  Since $k$ is a field, the first of the two isomorphisms in lemma $\autoref{lem:pre-weakWeibeliso}$, implies that $Ext_{U(\mathfrak{g})}^{\star}(k_{tr},M) \cong Ext_{U(\mathfrak{g})^e}^{\star}(U(\mathfrak{g}),Hom_k(k_{tr},M))$.  

By definition the universal enveloping algebra, of a Lie algebra is unital theore, the Hochschild cohomology groups of $U(\mathfrak{g})$ are isomorphic to the groups $Ext_{U(\mathfrak{g})^e}^{\star}(U(\mathfrak{g}),Hom_k(k_{tr},M))$.  Theore, so far we have $HL^{\star}(\mathfrak{g},M) \cong HH^{\star}(U(\mathfrak{g}),Hom_k(k_{tr},M))$.  However, in the expression the $U(\mathfrak{g})$-modules $k_{tr}$ and $M$ are viewed simply as their underlying $k$-modules.  Since $Hom_k(k,M)\cong M$ for any $k$-module $M$, then $HL^{\star}(\mathfrak{g},M) \cong HH^{\star}(U(\mathfrak{g}),Hom_k(k_{tr},M)) \cong HH^{\star}(U(\mathfrak{g}),M)$.  
\end{proof}

\subsection{The scarcity of quasi-free lie algebras}

The entire paper can be summarised within this lemma and its subsequent theorem:
\begin{lem} $\label{lem:isos}$
If $G$ is a compact complex Lie group, $\mathfrak{g}$ is its corresponding Lie algebra and $M$ is a $\mathfrak{g}$-module, then the following are isomorphic:

- $HH^{\star}(U(\mathfrak{g}),\mathbb{C})$

- $HL^{\star}(\mathfrak{g},\mathbb{C})$

- $H_{dr}^{\star}(G,\mathbb{C})$

- $H_{aDR}^{\star}(\mathbb{C}[G],\mathbb{C})$.  

If $G$ is of finite topological dimension $d$ then:

- $H_{\sigma d-\star}(G,\mathbb{C})$, is isomorphic to the cohomology theories above. 
\end{lem}
\begin{proof}
By proposition $\autoref{prop:lHoch}$ there are isomorphism $HH^{\star}(U(\mathfrak{g}),\mathbb{C}) \cong HL^{\star}(\mathfrak{g},\mathbb{C})$.  Now Cartan's theorem  gives the isomorphisms $HL^{\star}(\mathfrak{g},\mathbb{C}) \cong H_{dr}^{\star}(G,\mathbb{C})$.  Since $G$ was assumed to be compact, Serre's Gaga implies that $G$ corresponds to the proper affine scheme $Spec(\mathbb{C}[G])$ and so the affine version of Grothendieck's algebraic deRham theorem applies giving isomorphisms $H_{dr}^{\star}(G,\mathbb{C}) \cong H_{dr}^{\star}(Spec(\mathbb{C}[G])^{an},\mathbb{C}) \cong H_{dr}^{\star}(\mathbb{C}[G],\mathbb{C})$.  

Now, if $G$ is of finite topological dimension $d$, then the analytic deRham theorem gives the isomorphism $H_{aDR}^{\star}(G,\mathbb{C}) \cong H_{dr}^{\star}(G,\mathbb{C})$, following this up with Poincaré duality gives the desired isomorphism: $H_{dr}^{\star}(G,\mathbb{C}) \cong H_{aDR}^{\star}(\mathbb{C}[G],\mathbb{C})$.  
\end{proof}

\begin{prop} $\label{prop:lieDegen}$
If $G$ is a compact complex Lie group and $\mathfrak{g}$ is its corresponding Lie algebra, then both $G$'s coordinate ring $\mathbb{C}[G]$ and $U(\mathfrak{g})$ fail to be quasi-free \textit{if} there exists an integer $n\geq 2$ such that the Lie algebra comohology group $HL^n(\mathfrak{g},\mathbb{C})$ is non-trivial.  
\end{prop}
\begin{proof}
\textbf{1)}  Using lemma $\autoref{lem:isos}$ we identify the Lie algebra cohomology of $\mathfrak{g}$ with coefficients in $\mathbb{C}$ with $G$'s deRham cohomology $H_{dr}^{\star}(G,\mathbb{C})$.  

If such an integer $n$ exists then, since the group $H_{dr}^{n}(\mathbb{C}[G],\mathbb{C})$ then this would imply that the kernel $ker(\partial^{n+1})$ of the co-boundary map $\partial^{n+1}: \Omega_{\mathbb{C}[G]|\mathbb{C}}^n \rightarrow \Omega_{\mathbb{C}[G]|\mathbb{C}}^{n+1}$ is not the zero map.  Theore, the Kernel $Ker(\partial^{n+1})$ is a non-trivial sub-module of the module $\Omega_{\mathbb{C}[G]|\mathbb{C}}^n$ of $n$-forms; that is to say $\Omega_{\mathbb{C}[G]|\mathbb{C}}^n \not\cong 0$.

Since $G$ is a Lie group it is has underlying structure of a complex manifold.  Theore, it is the analytification of a smooth scheme, and so $\mathbb{C}[G]$ must theore be a smooth commutative algebra.  Now the \textit{Hochschild-Kostant-Rosenberg theorem}, implies that the module $\Omega_{\mathbb{C}[G]|\mathbb{C}}^n$, of Khäler $n$-forms on $\mathbb{C}[G]$ is exactly isomorphic to the Hochschild cohomology $\mathbb{C}$-module $HH^{n}(\mathbb{C}[G],\mathbb{C})$.  Theore, $HH^{n}(\mathbb{C}[G],\mathbb{C})$ is non-trivial.  Since the Hochschild cohomological dimension is the supremum, the extended natural numbers such that $HH^{n}(\mathbb{C}[G],\mathbb{C}) \not\cong 0$, then the Hochschild cohomological dimension of $\mathbb{C}[G]$ must be bounded below by $n$, which itself is by hypothesis atleast $2$.  The degeneracy criterion then affirms that $\mathbb{C}[G]$ cannot be quasi-free.  

\textbf{2)}  Now if the Lie algebra cohomology group $HL^n(\mathfrak{g},\mathbb{C})$ is non-trivial then, since by definition $U(\mathfrak{g})$ is unital and associative, there must exist natural isomorphisms $HL^n(\mathfrak{g},\mathbb{C}) \cong HH^{\star}(U(\mathfrak{g}),\mathbb{C})$ as described in lemma $\autoref{lem:isos}$.  This implies that, $HH^{\star}(U(\mathfrak{g}),\mathbb{C})$ is nontrivial for some $n\geq 2$.  Therefore the Hoschild cohomological dimension of $U(g)$ must be atleast $2$ and so by Quillin and Cuntz's characterisation of \textit{quasi-freeness}, $U(g)$ cannot be quasi-free.  
\end{proof}

The above result is a very powerful example of the lack of quasi-freeness of most smooth-algebra.  However, it farreachingness is shadowed by its lack of straightforward computability.  In my opinion a truly beautiful mathematical result, is one with a vastly rich and deep theoretical framework, contrasted by its totally concrete interpretation and almost trivially simple applicability.  In my opinion this is one such result, and essentially the theoretical closing-point for this memoir.  

\begin{thrm} $\label{thrm:liescarce}$
\textbf{The scarcity of quasi-free lie algebras}

If $G$ is a compact complex Lie group of finite topological dimension at least $2$, then both $G$'s coordinate ring $\mathbb{C}[G]$ and the universal enveloping algebra $U(\mathfrak{g})$, fail to be quasi-free.
\end{thrm}
\begin{proof}
Since $G$ is of a topological manifold of dimension atleast $2$ then it has atleast one connected component and so the first singular homology functor $H_{\sigma 0}(G,\mathbb{C})\cong \mathbb{C}\otimes_{\mathbb{Z}} H_{\sigma 0}(G,\mathbb{Z}) \not\cong \mathbb{C}\otimes_{\mathbb{Z}} 0 \cong 0$ is non-trivial (Spanier, ).  

Since $G$ is a complex manifold of \textit{finite} dimension $d$, then we apply the analytic Poincaré-lemma (Warner, 155), and so we have that the $d^{th}$ singular cohomology group of $G$ with coefficients in  $\mathbb{C}$ $H_{\sigma}^d(G,\mathbb{C})$, is isomorphic to $H_{\sigma 0}(G,\mathbb{C})$ which was seen to be non-trivial.  Theore, $H_{\sigma}^d(G,\mathbb{C})\not\cong 0$.

Now lemma $\autoref{lem:isos}$ implies $HL^{d}(\mathfrak{g},\mathbb{C})$ is non-trivial since it isomorphic to the non-trivial group $H_{\sigma}^d(G,\mathbb{C})$.  Since $G$ is a compact complex Lie group, proposition $\autoref{prop:lieDegen}$ implies that both $\mathbb{C}[G]$ and $U(\mathfrak{g})$ cannot be quasi-free.  Concluding the proof.  
\end{proof}

\subsection{Examples}
\subsubsection{Foreword}
The examples in this section, will be built-up from near-complete scratch and followed through till the complete end, making explicit use of virtually all the theory of this paper, and some lie theory beyond it. 

The first example will treat the case of a very familiar and centrally important affine algebraic group.  This group infact does fail to be compact theore theorem $\autoref{thrm:liescarce}$ will not be applicable.  However it is a fine oppertunity to approach the problem with other techniques illustrating both the practical use of theorem $\autoref{prop:wkDC}$ as well as showcasing some applications of some other properties of the Hochschild cohomology.  

The second example is presented to demonstrate not only how most complex Lie groups fail to be quasi-free, but the simplicity of application and general usefulness of theorem $\autoref{thrm:liescarce}$.

\subsubsection{Two thorough examples}
\begin{ex}
The coordinate ring of the Lie Group $GL_n(\mathbb{C})$ cannot be quasi-finite for values of $n$ strictly greater than $1$, likewise for the enveloping algebra of $GL_n(\mathbb{C})$'s lie algebra $\mathfrak{gl}_n(\mathbb{C})$.  
\end{ex}

\begin{proof}
Identify the complex manifold $M_n$, of all complex valued invertible $n$ by $n$ matrices can be identified with the complex manifold $\mathbb{C}^{n^2}$, via the set mapping $f$ taking the matrix $(a_{i,j})_{i,j=1}^n$ to the $n^2$-tuple $(a_{1,1},..,a_{1,n},...,a_{n,n})$.  We define the topology on $M_n$ as the weakest topology making the function $f$ continuous.  Moreover, since $\mathbb{C}^{n^2}$ admits a single chart, then pre-composing with $f$ defines a single chart on $M_n$ and so $M_n$ is diffeomorphic to $\mathbb{C}^{n^2}$.  

Now, since the determinant $det$ is a polynomial, it is a holomorphic function from $M_n$ to $\mathbb{C}$.  We may view the Lie group $GL_n(\mathbb{C})$ of all $n\times n$ invertible matrices, as the inverse image of the open set $\mathbb{C}_\{ 0 \} $ under $det$, since a matrix is invertible if and only if its determinant is non-zero.  Theore, $det$ defines a single chart from $GL_n(\mathbb{C})$ to $\mathbb{C}^{n^2}$.  By definition of the topological dimension of a manifold, $GL_n(\mathbb{C})$ is of the same topological dimension its chart's image; that is $GL_n(\mathbb{C})$ is of dimension $n^{2}$.    

So $GL_n(\mathbb{C})$ has inherit group structure, when viewed from within $M_n$, theore $GL_n(\mathbb{C})$ is indeed a Lie Group of dimension $n^2$.  However, $GL_n(\mathbb{C})$ is an \textit{open} subset of $\mathbb{C}^{n^2}$ theore, it cannot be compact if $n$ is non-zero.  So theorem $\autoref{thrm:liescarce}$ does not apply however, we may still approach the problem via the \textit{weak degeneracy criterion}.

First lets deal with its coordinate ring, when viewed as a scheme.  The complex manifold $C^{n^2}$ is by definition the analytification of the affine scheme $Spec(\mathbb{C}[x_1,...,x_{n^2}])$.  Since the determinant function is a smooth function between the complex manifolds $C^{n^2}$ an $\mathbb{C}$, then the full-functoriality of analitification implies that $det$ corresponds to a morphism between the schemes $Det: Spec(\mathbb{C}[x_1,...,x_{n^2}]) \rightarrow Spec(\mathbb{C}[x])$.  Theore, $GL_n(\mathbb{C})$'s corresponding affine scheme which we dentoe $GL_n$, is the inverse image of the principle open set $D(x)$ and so by the continuity of $Det$ is itself an open set.  In-fact, this implies that $GL_n$ is the Zarsiky open set on which the function $Det$ does not vanish, that is $GL_n$ is exactly the principle open set $D(Det)$.  

By definition of the Zarisky topology, this implies that $GL_n$ is the affine scheme corresponding to the localisation of $\mathbb{C}[x_1,...,x_{n^2}]$ at the function $Det$, that is $GL_n(\mathbb{C}) = Spec(\mathbb{C}[x_1,...,x_{n^2}]_{Det})^{an}$.  Affine $n^2$-space is smooth, and since $GL_n$ is an open subset of a smooth scheme it must also be smooth.  Theore, $GL_n$'s coordinate ring $\mathbb{C}[x_1,...,x_{n^2}]_{Det}$ must be regular since $\mathbb{C}$ is an algebraically closed field.  The correspondence between prime ideals in $\mathbb{C}[x_1,...,x_{n^2}]$ and prime ideals in $\mathbb{C}[x_1,...,x_{n^2}]_{Det}$ not intersecting the ideal $(Det)$ implies that if $n \geq 2$, then the chain of ideals $(x_1) \lhd (x_1,x_2)$ is corresponds to a chani of prime ideals in $\mathbb{C}[x_1,...,x_{n^2}]_{Det}$.  Theore, $\mathbb{C}[x_1,...,x_{n^2}]_{Det}$ is of Krull dimension at least $2$, and we find ourselves in scenario $2$ of the \textit{weak degeneracy criterion} (theorem $\autoref{prop:wkDC}$), taking care of the claim on the coordinate ring of $GL_n$.  

Now we address the claim on the universal enveloping algebra of $G_n(\mathbb{C})$'s lie algebra, lets call that Lie algebra $\mathfrak{gl}_n(\mathbb{C})$.  Since $GL_n(\mathbb{C})$ is of finite dimension $n^2$, lemma $\autoref{lem:isos}$ implies that existence of the isomorphisms $H_{dr}^{\star}(U\mathfrak{gl}_n(\mathbb{C}),\mathbb{C}) \cong H_{\sigma n^2-\star}(GL_n(\mathbb{C}),\mathbb{C})$.  Since $GL_n(\mathbb{C})$ is not an empty point-set, Pointcare duality implies $H_{\sigma}^{n^2-\star}(GL_n(\mathbb{C}),\mathbb{C}) \cong \mathbb{C}$ and so if $n\geq 2$, then Cuntz and Quillin's characterisation of quasifreeness implies that $U(gl_n(\mathbb{C}))$ cannot be quasi-free, since by lemma $\autoref{lem:isos}$ $HH^{n^2}(U(\mathfrak{gl}_n(\mathbb{C})), \mathbb{C}) \cong H_{\sigma}^{n^2-\star}(GL_n(\mathbb{C}),\mathbb{C}) \cong \mathbb{C} \not \cong 0$.

We finish off this example by giving an explicit description of 

$U(\mathfrak{gl}_n(\mathbb{C}))$, so we know what we've dealt with.  

Now $GL_n$ is the exponential of a subset of a certain matrices $n$ by $n$ complex-valued matrices.  The only criterion for a matrix being in $GL_n(\mathbb{C})$ is that its determinant is non-trivial.  Theore if a matrix $X$ exponentiates into $GL_n(\mathbb{C})$, then it must satisfy $0\neq det(e^X) = e^(tr(X))$, but the exponent of any real number cannot be zero. Theore, this criterion is verified by any $n$ by $n$ matrix and so as a $\mathbb{C}$-vector space $gl_n(\mathbb{C})$ and $M_n$ are identical.  Now the commutator gives the Lie bracket on $gl_n(\mathbb{C})$ making it into a Lie algebra.  

\end{proof}

\begin{ex}
The coordinate ring of the Unitary Group cannot be quasi-finite for values of $n$ strictly greater than $1$, likewise for the enveloping algebra of $GL_n(\mathbb{C})$'s lie algebra $\mathfrak{gl}_n(\mathbb{C})$.  
\end{ex}
\begin{proof}
The Unitary group $U_n(\mathbb{C})$ is the subgroup of $GL_n(\mathbb{C})$ consisting of all invertible complex-valued unitary matrices.  

Viewing the entries of a matrix $M =(m_{i,j})$, in $U_n(\mathbb{C})$ as real values $2n\times 2n$-matrices of the form 
$
\begin{pmatrix}
  a_{i,j} & -b_{i,j}  \\
  b_{i,j} & a_{i,j}  \\
 \end{pmatrix}
 $
 , where $a_{i,j} +b_{i,j} = m_{i,j}$, we may describe the complex conjugation operation as an $2n\times 2n$-real matrix $C$ of the form $C=(c_{i,j})$
 $\begin{pmatrix}
  0_2 & ...& 0_2 & 0_2 & \tilde{C}  \\
  0_2 & ...& 0_2 & \tilde{C} & 0_2  \\
  0_2 & ...& \tilde{C} & ...& 0_2\\
  0_2 & \tilde{C} & 0_2 & ... & 0_2  \\
  \tilde{C} & 0_2 & ...& 0_2 & 0_2  \\
 \end{pmatrix}$,  
  where $\tilde{C}$ is the $2$ by $2$ real matrix: 
 $
 \begin{pmatrix}
  0 & 1  \\
  1 & 0  \\
 \end{pmatrix}
 $, and 
 $0_2$ is the $2$ by $2$, zero matrix.  
 
Now a matrix $M$ is unitary if and only if its inverse is exactly is conjugate transpose, that is if and only if $MCM=I_{2n}$, where $I_{2n}$ denotes the real $2$ by $2$ identity matrix and $M$ is an element of $GL_n(\mathbb{C})$.  That is, $M$ is unitary if and only if 

the entries of $M$ satisfy the polynomial equations $\underset{j,k}{\sum}m_{i,j}c_{j,k}m_{k,l}=\delta_{i,l}$ for each $i,l \in \{ 1,...,2n\}$, where $\delta_{i,j}$ is the Kronecker delta.  

More compactly put, $M$ must satisfy the single polynomial equation $P=0$, $P$ is the complex polynomial $\underset{i,l}{sum}\underset{j,k}{\sum}m_{i,j}c_{j,k}m_{k,l}-\delta_{i,l}=0$.  

In other words, a matrix $M$ in $GL_n(\mathbb{C})$, as identified with some open subset of $\mathbb{R}^{(2n)^2}$, is an element of $U_n(\mathbb{C})$ if and only if $M$ is in the intersection of the analytic zero-locus of the polynomial $\underset{j,k}{\sum}m_{i,j}c_{j,k}m_{k,l}=\delta_{i,l}$, in other words $U_n(\mathbb{C})$ an analytic sub-variety of $ GL_n(\mathbb{C})$.  

Now, for any matrix $M$ in $U_n(\mathbb{C})$, its determinant satisfies $Det(I_{2n})=Det(MCM)=Det(M)Det(C)Det(M)=Det(M)^2$, as viewed in $\mathbb{R}^{(2n)^2}$, this is the collection of points who's norm is $1$; so the determinant function restricts to a surjection onto the complex $n$-sphere $S^n(\mathbb{C})$.  Since $det$ was a homeomeprhims on $GL_n(\mathbb{C})$ its restriction becomes a homeomrphism onto the complex manifold $S^n(\mathbb{C})$ of dimension $n$.  So $U_n(\mathbb{C})$ is a complex manifold of dimension $n$, moreover since $det$ is a homeomorphism of $U_n(\mathbb{C})$ onto a compact subspace of $\mathbb{C}^n$, $U_n(\mathbb{C})$ must be a compact lie group of dimension $n$.  

Now if $M$ and $N$ are matrices in $U_n(\mathbb{C})$ 

$(MN^{-1}C)MN^{-1} = N^{-1}C(MCM)N^{-1}$ 

$= N^{-1}CN^{-1}=(NCN)^{-1}=I_{2n}^{-1}=I_{2n}$, and theore $U_n(\mathbb{C})$ is a closed subgroup of $GL_n(\mathbb{C})$, and so it is a compact Lie group of dimension $n$.  Theorem $\autoref{thrm:liescarce}$ theore applies, and so $U_n(\mathbb{C})$'s corresponding affine scheme's coordinate ring and its so is the universal enveloping algebra of its lie algebra are both singular for values of $n$ strictly greater than $1$.

We now wish to explicitate both these $\mathbb{C}$-algebras, for clarity.  
By describing $U_n(\mathbb{C})$ as an analytic variety, we explicitly described the coordinate ring of its corresponding affine scheme which we denote $U_n$.  That is $U_n$ is exactly the affine subscheme $Spec(\mathbb{C}[X_1,..,X_{n^2}]_{det}/P)$ of $GL_n$.  Theore the polynomial ring $\mathbb{C}[X_1,..,X_{n^2}]_{det}/P$ is not quasi-free.  

Now, to describe the universal enveloping algebra of $U_n(\mathbb{C})$'s lie algebra.  If $X$ is in $\mathfrak{gl}_n(\mathbb{C})$, and its exponent is in $U_n(\mathbb{C})$ then it follows that $e^0=1e^{X^{\star}}e^{X}=e^{X^{\star}+X}$; that is the lie algebra of $U_n(\mathbb{C})$ is the collection of matrices in $gl_n(\mathbb{C})$, satisfying $0 = X^{\star} + X$.  We denoted this collection of matrices $\mathfrak{u}_n(\mathbb{C})$.  No, since $U_n(\mathbb{C})$ is a subgroup of $GL_n(\mathbb{C})$, then $\mathfrak{u}_n(\mathbb{C})$ is a lie subalgebra of $\mathfrak{gl}_n$.  

Now $U_n(\mathbb{C})$ is a complex manifold, theore in particular it is smooth, theore its tangent space is of the same dimension as itself.  Identifying its tangent space with its lie algebra $\mathfrak{u}_n(\mathbb{C})$ we know that $\mathfrak{u}_n(\mathbb{C})$ must be spanned as a $\mathbb{C}$-vector space by exactly $n$ elements.  In particular we notice, that the matrices $\{iE_{r,r} \}_{r=1}^n$ having a $i$ in the $(i,i)^{th}$ coordinate and $0$s in everywhere else are equal to $(-1)$ times their own conjugate transpose and so each $iE_{r,r}+iE_{i,r}C=iE_{r,r} - iE_{r,r}=0$.  Theore, $\{iE_{r,r} \}_{i=1}^n$ forms a maximal linearly independent set of matrices in $\mathfrak{u}_n(\mathbb{C})$, and so $iE_{r,r}$ describes a basis for the underlying $\mathbb{C}$-vector space of $\mathfrak{u}_n(\mathbb{C})$.  

Ordering this basis in an expected way as $iE_{r,r}<iE_{r+1,r+1}$, and applying the \textit{Poincare-Birkhoff-Witt theorem} (See: Knapp, 217) we see that $U(\mathfrak{u}_n(\mathbb{C}))$ is the subalgebra of $U(\mathfrak{gl}_n(\mathbb{C}))$, spanned as a $\mathbb{C}$-vector space by the monomials of the form $\iota(iE_{r_1,r_1})^j_1<...<\iota(iE_{r_n,r_n})^j_k$ with and $j_l,..,j_n$ natural numbers.  And with multiplication, consistent with the commutator bracket in $\mathfrak{u}_n(\mathbb{C})$.  

\end{proof}

Though, this case could have been treated with the weak degeneracy criterion, it was much simpler and more illustrative to approach the problem using theorem $\autoref{thrm:liescarce}$.  
\chapter{The Essentials}
Since every commuttive algebra of Krull dimension atleast $2$
\begin{defn}
\textbf{Essentially Unramified}

Let $f:A\rightarrow B$ be a morphism of $R$-algebras, then $X$ is called \textbf{essentially unramified} if and only if $f^{\varsigma}$ is formally unramified.  
\end{defn}

\begin{defn}
\textbf{Essentially Smooth}

Let $f:A\rightarrow B$ be a morphism of $R$-algebras, then $X$ is called \textbf{essentially smooth} if and only if $f^{\varsigma}$ is formally smooth.  
\end{defn}

Similarly we may define the following analogous notion:
\begin{defn}
\textbf{Essentially Etale}

Let $f:A\rightarrow B$ be a morphism of $R$-algebras, then $X$ is called \textbf{essentially étale} if and only if $f^{\varsigma}$ is formally etale.  
\end{defn}

The following characterizes essential smoothness.  

\begin{prop} $\label{prop:equiv}$
The following are equivalent:

- $f:A\rightarrow B$ is essentially smooth

- $\Omega_{B^{\varsigma} | A^{\varsigma}}$ is a projective $B^{\varsigma}$-bimodule. 
\end{prop}
\begin{proof}
$f: A \rightarrow B$ is essentially smooth if and only if $f^{\varsigma}:A^{\varsigma} \rightarrow B^{\varsigma}$ is formally smooth.

Now a morphisms of commutative associative unital $R$-algebras is formally smooth \textit{(stacks tag: 049R)} if and only if the module of Khaler differentials $\Omega_{HH_{1}(B^{\varsigma},A^{\varsigma}) | HH_{1}(A^{\varsigma},A^{\varsigma})}$ is a projective $B^{\varsigma}$.  
\end{proof}

In particular we may think of an $R$-algebra as \textit{essentially smooth} if the inclusion $0 \rightarrow A$ is an \textit{essentially smooth} $R$-algebra homomorphism.  In this case we may relate things back to to Hochschild homology as follows:

\begin{prop}
$A$ is \textit{essentially smooth} only if $HH^1(A^{\varsigma})$ is a projective $A$-bimodule. 
\end{prop}
\begin{proof}
By definition, $A$ is essentially smooth if and only if the $R$-algebra monomorphisms $i:0^{\varsigma} \rightarrow A^{\varsigma}$ is formally smooth.  

Now $0^{\varsigma}$ is simply isomorphic to the commutative unital associative $R$-algebra, $R$.
Therefore, $\Omega_{A^{\varsigma} | 0^{\varsigma}}\cong \Omega_{A^{\varsigma} | R}$ is a projective $A^{\varsigma}$-bimodule.  Since $A^{\varsigma}$ is formally smooth, then the \textit{Hochschild-Kostant-Rosenberg theorem} \textit{(arXiv:math/0506603)} is applicable giving the conclusive isomorphism $HH^1(A^{\varsigma})\cong \Omega_{A^{\varsigma} | R}$, of $A^{\varsigma}$-bimodules.  
\end{proof}

\subsection{The main interest in essential smoothness}
The main problem with Qullin and Cuntz's non-commutative analogue of formal smoothness, that is quasi-freeness, is far too restrictive.  For example, very regular $\mathbb{C}$-algebra becomes singular as soon as it is of dimension at Krull least $2$, since its Krull dimension is equal to $A$'s Global dimension as a ring which can easily be seen to bound the Hochschild cohomological dimension of $A$ below.  

Geometrically, this seems inappropriate in many ways, since the passage from $_RAlg$ to $_R\mathfrak{Alg}$, should be beneficial and not a limitation.  That is, geometrically speaking, a "good" generalisation of formal smoothness, should infact be a generalization and a restriction of the initial notion.  This motivates the central result:

\begin{thrm} $\label{thrm:lethrm}$
If $f:A \rightarrow B$ is a morphism in $_RAlg$ as viewed in $_R\mathfrak{Alg}$, then:

1) $f$ is formally smooth in $_RAlg$ \textit{if and only if} $f$ is essentially smooth as viewed in $_R\mathfrak{Alg}$.  

2) $f$ is formally unramified in $_RAlg$ \textit{if and only if} $f$ is essentially unramified as viewed in $_R\mathfrak{Alg}$.  

3) $f$ is formally etale in $_RAlg$ \textit{if and only if} $f$ is essentially etale as viewed in $_R\mathfrak{Alg}$.  
\end{thrm} $\label{thrm}$
First we proove a little lemma:
\begin{lem} $\label{lem:FEt}$
For any $R$-algebra $A$ in $_RAlg$, the identity $R$-algebra endomorphism $1_A:A \rightarrow A$, is formally etale; particularly $1_A$ is formally smooth.  
\end{lem}
\begin{proof}
If there is an $R$-algebra homomorphism $f:A \rightarrow B$ and $I$ is an ideal in $B$ then trivially, if:

$
\begin{tikzpicture}
  \matrix (m) [matrix of math nodes,row sep=3em,column sep=4em,minimum width=2em]
  {
     A & B \\
     A & B/I \\};
  \path[-stealth]
    (m-1-1) edge node [left] {$1_A$} (m-2-1)
            edge node [below] {$f$} (m-1-2)
    (m-2-1.east|-m-2-2) edge node [below] {$g$}
            node [above] {$$} (m-2-2)
    (m-1-2) edge node [right] {$\pi_I$} (m-2-2);
\end{tikzpicture}
$

where to be a commutatie diagram in $_RAlg$ with $\pi_I$ the cannonical projection of $B$ onto the quotient $R$-algebra $B/I$.  Then there by the univeral property of the quotient algebra, $g$ is the \textit{unique} $R$-algebra homorphism from $A$ to $B/I$ such that $f\circ \pi_I = g$.  In other words, $f$ is the unique lift from $A$ to $B/I$.  In particular if $I$ is square-zero then $1_A$ is formally etale (and therefore also formally unramified and formally smooth).  
\end{proof}

\textit{We prove scenario $1$, leaving the the others to the reader, since their proof differs mutatis mutandis, by one quantifier.  }

\begin{proof}

Suppose $f: A \rightarrow B$ is a formally smooth $R$-algebra morphism \textit{(in $_RAlg$)}, as viewed within $_R\mathfrak{Alg}$.  And consider a commutative diagram in $_RAlg$ with $I$ square-zero:

$
\begin{tikzpicture}
  \matrix (m) [matrix of math nodes,row sep=3em,column sep=4em,minimum width=2em]
  {
     A^{\varsigma} & B \\
     C^{\varsigma} & B/I \\};
  \path[-stealth]
    (m-1-1) edge node [left] {$f^{\varsigma}$} (m-2-1)
            edge node [below] {$h$} (m-1-2)
    (m-2-1.east|-m-2-2) edge node [below] {$g$}
            node [above] {$$} (m-2-2)
    (m-1-2) edge node [right] {$\pi_I$} (m-2-2);
\end{tikzpicture}
$

As it was noted above, since any unit preservign $R$-algebra homorphisism $\psi: A \rightarrow R \rightarrow D$ bewteen unital $R$-algebras, can be written uniquely as the direct sum of two $R$-algebra morphisms $\tilde{\phi}$ and $1_R$.  In particular, this is the case for the $R$-algebra homorphisms $f^{\varsigma}$, $h$ and $g$ in the diagram above, which we identify with the $R$-algebe homorphisms $\tilde{f^{\varsigma}}:A \rightarrow B \oplus 1_R$, $\tilde{h}:A \rightarrow B \oplus 1_R$ and $\tilde{g}:C \rightarrow B/I \oplus 1_R$, respectivly.  The uniqueness of this decomposition of the morphisms, then implies that $\tilde{f^{\varsigma}} = f \oplus 1_R$.  

By hypothesis $f$ is formally smooth in $_RAlg$.  Therefore, there exists a $R$-algebra homorphism $\phi:C \rightarrow B$ making the diagram:

$
\begin{tikzpicture}
  \matrix (m) [matrix of math nodes,row sep=3em,column sep=4em,minimum width=2em]
  {
     A & B \\
     C & B/I \\};
  \path[-stealth]
    (m-1-1) edge node [left] {$f$} (m-2-1)
            edge node [below] {$\tilde{h}$} (m-1-2)
    (m-2-1.east|-m-2-2) edge node [below] {$\tilde{g}$}
            node [above] {$ $} (m-2-2)
    (m-1-2) edge node [right] {$\pi_I$} (m-2-2)
    (m-2-1) edge [dashed] node [below] {$\phi$} (m-1-2);
\end{tikzpicture}
$ commute.  

Moreover, by lemma $\autoref{lem:FEt}$, we know there exists a unique $R$-algebra homorphism $\phi_1$ from $R$ to $B$ making the diagram:
$
\begin{tikzpicture}
  \matrix (m) [matrix of math nodes,row sep=3em,column sep=4em,minimum width=2em]
  {
     R & B \\
     R & B/I \\};
  \path[-stealth]
    (m-1-1) edge node [left] {$1_R$} (m-2-1)
            edge node [below] {$1_R$} (m-1-2)
    (m-2-1.east|-m-2-2) edge node [below] {$1_R$}
            node [above] {$ $} (m-2-2)
    (m-1-2) edge node [right] {$\pi_I$} (m-2-2)
    (m-2-1) edge [dashed] node [below] {$\phi_1$} (m-1-2);
\end{tikzpicture}
$ commute.  

Taking the direct sum of the morphisms $\phi$ and $\phi_1$ we get a map, $\psi: C \oplus R \cong C^{\varsigma} \rightarrow B$ making the initial diagram:

$
\begin{tikzpicture}
  \matrix (m) [matrix of math nodes,row sep=3em,column sep=4em,minimum width=2em]
  {
     A^{\varsigma} & B \\
     C^{\varsigma} & B/I \\};
  \path[-stealth]
    (m-1-1) edge node [left] {$f^{\varsigma}$} (m-2-1)
            edge node [below] {$h$} (m-1-2)
    (m-2-1.east|-m-2-2) edge node [below] {$g$}
            node [above] {$$} (m-2-2)
    (m-1-2) edge node [right] {$\pi_I$} (m-2-2);
\end{tikzpicture}
$ 
commute.  Since the diagram was chosen arbitrarily, then $f^{\varsigma}$ is formally smooth in $_RAlg$ and so $f$ is essentially smooth.  

Conversly, if $f$ a morphism in $_RAlg$ as viewed in $_R\mathfrak{Alg}$ which is essentially smooth, then by definition $f^{\varsigma}=f \oplus 1_R$ is formally smooth in $_RAlg$.  By the lemma $1_R$ is formally etale \textit{(and therefore also formally smooth)}.  Now if there would exist a diagram:

$
\begin{tikzpicture}
  \matrix (m) [matrix of math nodes,row sep=3em,column sep=4em,minimum width=2em]
  {
     A & B \\
     C & B/I \\};
  \path[-stealth]
    (m-1-1) edge node [left] {$f$} (m-2-1)
            edge node [below] {$h$} (m-1-2)
    (m-2-1.east|-m-2-2) edge node [below] {$g$}
            node [above] {$$} (m-2-2)
    (m-1-2) edge node [right] {$\pi_I$} (m-2-2);
\end{tikzpicture}
$ in $_RAlg$, with $I$ square-zero.  Such that there does not exist a morphism from $C$ to $B$ lifting making the diagram commuate, then taking direct sums with $1_R$ would imply that the diagram:

$
\begin{tikzpicture}
  \matrix (m) [matrix of math nodes,row sep=3em,column sep=4em,minimum width=2em]
  {
     A^{\varsigma} & B \\
     C^{\varsigma} & B/I \\};
  \path[-stealth]
    (m-1-1) edge node [left] {$f^{\varsigma}$} (m-2-1)
            edge node [below] {$h\oplus 1_R$} (m-1-2)
    (m-2-1.east|-m-2-2) edge node [below] {$g\oplus 1_R$}
            node [above] {$$} (m-2-2)
    (m-1-2) edge node [right] {$\pi_I$} (m-2-2);
\end{tikzpicture}
$

would be an example of a diagram in $_RAlg$ for which $I$ is square-zero and there would not exist a morphism in $_RAlg$ from $C^{\varsigma}$ to $B$ allowing for the commutativity of the diagram.  Contradicting the formal smoothness of $f^{\varsigma}$ and therefore the essential smoothness of $f$.  Hence, if $f$ is an $_RAlg$ morphism which is essentially as viewed within $_R\mathfrak{Alg}$, then it must have been formally smooth in $_RAlg$.  

\end{proof}

In-fact, quasi-freeness, is generally a much stronger version of essential smoothness.  
\begin{cor}
If $R$ is a hereditary ring and $A$ is a commutative unital quasi-free $R$-algebra, then $A$ is essentially smooth.  
\end{cor}
\begin{proof}
Quillin and Cuntz characterised the quasi-freeness of an $R$-algebra $A$, as the projectivenss of its $A$-bimodule of non-commutative differential 1-forms $\Omega (A)$ being projective, in their paper \textit{"Algebra Extensions And Nonsingularity"}.  

Since, $\Omega^1 (A)$ quotiented by its sub-module $M$, consisting of all the symbols of the form $xdy-ydx$ or $dxdy-dydx$, for any $x,y$ in $A$ is isomorphic to the $A$-bimodule $\Omega_{A|R}$ of Khaler 1-forms and $R$ is hereadity, then $\Omega_{A|R}$ must be $A$-projective.  Then, by proposition $\autoref{prop:equiv}$, $A$ must be smooth in $_RAlg$ and therefore by theorem $\autoref{thrm}$ $A$ must be essentially smooth in $_R\mathfrak{Alg}$.  

\end{proof}

For these reasons, it seems reasonable to consider this sort of relativised formally smoothness, which we called \textit{essential smoothness}, as it is a proper generalisation of the notion of formally smoothness which only gains in morphisms and looses none.  Moreover, it is has similarities to the notion of quasi-freeness imposed on $_R\mathfrak{Alg}$.

\section{Free-esque-Examples}

Directly from theorem $\autoref{thrm}$ it follows that the algebras $R[X_i]_{i\in I}$, for some indexing set $I$ are essentially smooth in $_RAlg$, which was remarked to not be quasi-free in the case where $R$ was $\mathbb{C}$ and $I$ is of cardinality at least $2$, by Quillin and Cuntz.  

Now, the free $R$-algebra $R<X_i>_{i\in I}$'s abelianization is the algebra $R[X_i]_{i\in I}$.  Now similar, remarks to the above imply that $R[X_i]_{i\in I}$ is essentially smooth, which will imply that $R<X_i>_{i\in I}$ is essentially smooth also.  Which is geometrically intuitive.

\section{Essential Localizations}
One of the most powerful tools of this theory, is one that allows for a very natural topoogification of the category $_R\mathfrak{Alg}$.  This is the same concept as before, but now applied to a morhpism's standardization in $_R\mathfrak{Alg}$ that becomes the universal arrow of a localization in $_{Z(R)}Alg$.  

From this all the classical theory may be pullback to $_R\mathfrak{Alg}$ at heart, such as Zariksy topologies open and closed subschemes.  The purpous of this little interlude is not at all to replace any concepts of localization of Jackobson topologies on $_R\mathfrak{Alg}$, but to simply provide another tool possesing certain commutative-remenisent properties.  

\subsection{A familiar face}
\begin{defn}
\textbf{Essential Localization}

A morphism $\nu :A \rightarrow B$ is said to be an \textbf{essential localization} \textit{if and only if} $\nu^{\varsigma}: A^{\varsigma} \rightarrow B^{\varsigma}$ is the universal arrow of a localization of $A^{\varsigma}$ on some multiplicativly closed subset $S$ of $A$ in .  

$
\begin{tikzpicture}
  \matrix (m) [matrix of math nodes,row sep=3em,column sep=4em,minimum width=2em]
  {
     A^{\varsigma} & B^{\varsigma} \\
      & S^{-1}[A] \\};
  \path[-stealth]
    (m-1-1) edge node [above] {$f^{\varsigma}$} (m-1-2)
            edge node [below] {$\psi\circ f^{\varsigma}$} (m-2-2)
    (m-1-2) edge [dashed,->] node [left] {$\exists$} (m-2-2);
\end{tikzpicture}
$

That is there exists a mutliplicativly closed subset $S$ of $A$, and an isomorphism $\psi: S^{-1}[A] \rightarrow B$ making the above diagram, in $_{Z(R)}Alg$ commute.  

If we wish to explicitly make reference to the mutliplicativly closed subset $S$ of $A^{\varsigma}$ giving rise to the isomorphism $\psi$ described above, say that the essential localization $\nu$, is \textbf{essentially a localization of $A$ at S}.  
\end{defn}
Reminishent of classical notations, we write $S^{-1}[A]$ for the essential localization of $A$ at $S$, and if $S= \{ f \}$ is a one element set we write $A_f$ for the essential localization of $A$ at $S$.  

By definition of essential localization, the mutliplicativly closed subset $S$ of $A^{\varsigma}$ is unique up to isomorphism.  Nevertheless lets put it down on papaer as a proposition.  

\begin{prop} $\label{prop:ULoc}$
If $\nu: A\rightarrow B$ is an essential localization, then there is a unique multiplicativly closed subset $S$ of $A^{\varsigma}$ such that $A$ is the localization of $A$ at $S$.  
\end{prop}
\begin{proof}
Follows from the universal propserty of localization in $_{Z(R)}Alg$.  
\end{proof}

Lets present a simple yet fun example, to illustrate this contempt's application to a strictly non-commutative non-unital algebra.  

\begin{ex}
Let $R$ be an integral domain and define $<x,y>_+$ to be the non-unital $R$-subalgebra of $Z(R)<x,y>$ who's elements are non-commutative polynomials of left and right degree atleast $1$.  Likewise, define $A$ to be the $R$-algebra, who's underlaying $Z(R)$-module structure is given by the symbols $\frac{f}{g}$, with $f,g \in <x,y>_+$ and where $(x-1)$ does not divide $y$ on the left or on the right and with multiplication $\frac{f_1}{g_1} \times \frac{f_1}{g_2} \mapsto \frac{f_1f_2}{g_1g_2}$.  Then this is an $Z(R)$-module morphism making $A$ into an object in $_R\mathfrak{Alg}$.  

Now the map, $\nu: <x,y>_+ \rightarrow A$ sending an element $f$ of $<x,y>_+$ to the symbol $\frac{f}{1}$'s standardization may be seen to be the localization map between the commutative unital associative $Z(R)$-algebras $Z(R)[x,y] \rightarrow Z(R)[x,y]_{(x-1)}$.  So $\nu$ is the essential localization of $\nu$ at the single element set $\{ (x-1) \}$.  
\end{ex}

\subsection{The Essential-Zarisky Topology}
Although, this would fit most formally in the geometry section of this paper, it seems more appropriate to follow up the definition of essential localization with probably one of its most powerful applications, the realization of a Grothendieck-covering on the opposite category of $_R\mathfrak{Alg}$.  In turn, this will be used in later a section, to define \textit{essential commutative schemes}.  

Here we think of objects in the category $_R\mathfrak{Alg}$ as analogues to affine schemes over the basering $R$.  The main concept here, is that an open immension of affine schemes which was the dual to the localization map of its coordinate ring, is generalized by the idea of the dual of an essential localization.  This dual morhpism, may be viewed as an open immersion of the affine schemes over $Z(R)$, corresponding to the standardization of its \textit{"coordinate rings"} in $_R\mathfrak{Alg}$.  

\begin{defn} $\label{def:EZar}$
\textbf{The Essential-Zarisky Cover}

A familly of morphism $\{ \phi_i : X_i \rightarrow Y^{op} \}$ in $_R\mathfrak{Alg}^{op}$ is called an \textbf{essential-Zarisky Covering} \textit{if and only if}:

-  There are isomorphisms $\psi_i :Y^{op}_{f_i} \rightarrow X_i$ in $_R\mathfrak{Alg}$, between $X_i^{op}$ and the essential localization of $Y^{op}$ at some one point set $\{ f_i \}$.  Moreover, $\phi_i^{op}\circ \psi_i$ is the\textit{(unique)} universal arrow in $_R\mathfrak{Alg}$, describing that essential localization.  

-  There exist elements $y_i$ in $(Y^{op})^{\varsigma}$ such that $\underset{i}{\sum} y_if_i = 1$.  
\end{defn}

Similarly to the usual setting, the first clause in the above defintion may be thought of as describing the usual Zarisky cover of the \textit{"commutative part"} of the non-commutative affine scheme corresponding to $Y^{op}$.  Likewise, the second clause descibes a partition of unity on the \textit{"commutative part"} of the non-commutative affine scheme corresponding to $Y^{op}$.  In the following section we formalize these ideas, but for now we let the reader digest the essence of the matter.  

We have shown the following result:
\begin{prop} $\label{prop:site}$
The pair $<_R\mathfrak{Alg}^{op},Cov>$ of the category $_R\mathfrak{Alg}^{op}$ and the collection $Cov$ of all Essentially-Zariski Coverings of objects in $_R\mathfrak{Alg}^{op}$ defines a site.  
\end{prop}
We denote this site by the same symbol as its underlaying category, that is we denoite it as $_R\mathfrak{Alg}^{op}$.

\part{Geometry}

\chapter{Geometric Interpretations}
\section{Smooth essence}
\subsubsection*{Foreword and directions}
The goal of this short section is to purpose illustrate geometrically why the concept of essential smoothness is a prefectly intuitive and natural idea and provides a strict generalisation of smoothness in the case of \textit{"commutative"} affine schemes.  

The motivations for this is, simple.  After defining non-commutative schemes we notice that every such object contains a unique largest \textit{"commutative"} subscheme; moreover this containment is \textit{always} proper.  That is, schemes \textit{"gain points"} when considered as non-commutative objects, intuitivly this is where the ubiquitous singularities arise, from when abiding by the past definitions of smoothness.  

This association will be seen to be functorial, and therefore if the largest commutative part of a scheme is smooth, we will say that the scheme is essentially smooth.  We then close with examples of objects that, intuitivly should be essentially smooth, finally we relate this back to the Hochschild homology of the \textit{"coordinate ring"} of a nc.scheme.

\subsection{nc. Schemes}
\subsubsection{Affine Nc. Geometry}
The first step is to formalise our construciton of a nc. scheme.  With good reason, we follow the popular convetion, and inspiration of \textit{Michel Demazure}'s construction of he functor of points corresponding to an affine scheme.  

Following in the not so far-removed wonderful footsteps of Grothenrieck, we approach things in their natural setting, that is the setting of a certain \textit{topoi}.  The case where the basering is non-commutative is very interesting however there are significantly nicer proerties that arise when the basering is commutative.

\begin{defn}
\textbf{Affine nc.Scheme}

A non-commutative affine scheme $X$ is a corepresentable functor from $\mathfrak{_RAlg}$ to $Sets$.  

In other words $X$ is a functor from $\mathfrak{_RAlg}$ to sets, which is \textit{naturally isomorphic} to a functor of the form $Hom_{\mathfrak{_RAlg}}(A,-)$, for some $R$-algebra $A$.  
\end{defn}

\begin{prop}
There is an embedding of the dual category of $\mathfrak{_RAlg}$ into $[\mathfrak{_RAlg}:Sets]$.  
\end{prop}
\begin{proof}
Since the collections of all morphisms between two non-commutative algebras forms a set, this is exactly a rephrasing of the (Contravariant) Yoneda lemma.  
\end{proof}

In view, of this last result, for any $R$-Algebra $A$, we denote its contravariant-Yoneda embedding $Y:A \mapsto Hom_{\mathfrak{_RAlg}}(A,-)$, by the usual $Y(A)$.  In other words, any affine nc. scheme is naturally equivalent to the contravariant-Yoneda embedding of an $R$-algebra into $[\mathfrak{_RAlg}:Sets]$.

\subsubsection{Projective Nc. Geometry}
Before being able to properly define an essentially nc. scheme lets peek at a little useful lemma.
\begin{lem} $\label{lem:precomp}$
Every functor $F: \mathfrak{C} \rightarrow \mathfrak{D}$, induces a functor $F^{\star}: [\mathfrak{D}:Sets] \rightarrow [\mathfrak{C}:Sets]$ by precomposing a functor $X$ in $[\mathfrak{D}:Sets] $ with $F$.  
\end{lem}
\begin{proof}
The composition of functors is again a functor, therefore if $F$ is a functor so must any precomposition with it be.  
\end{proof}

Particularly, the standardization functor $-^{\varsigma}:_R\mathfrak{Alg} \rightarrow_{Z(R)}Alg$ induces a functor, $(-^{\varsigma})^{\star}: [_R\mathfrak{Alg}:Sets] \rightarrow [_{Z(R)}Alg: Sets]$.  

This type of functor is a reoccuring construction and it so happends that it has a name.  
\begin{defn}
\textbf{Inverse Functor}

The \textbf{inverse functor} $F^{\star}$ of a functor $F$, is the functor $F^{\star}: [\mathfrak{D}:Sets] \rightarrow [\mathfrak{C}:Sets]$ induced by the functor $F:\mathfrak{C}\rightarrow \mathfrak{D}$.   $F^{\star}$ takes a natural transformation $\phi$ in $[\mathfrak{C}:Sets]$ and assigning to it the morphism $\phi (F(-)): A \mapsto \phi(F(A))$.  
\end{defn}

\begin{defn}
\textbf{Nc. Scheme}

A \textbf{nc. Scheme} is a sheaf $X$ on the site $_R\mathfrak{Alg}^{op}$ with the added property that:

There exists a familly of affine nc. schemes $\{ Y(A_i) \}_i$ in $[_R\mathfrak{Alg}:Sets]$ such that, on all $Z(R)$-algebras $B$ which are fields, the sets $(X^{\varsigma})^{\star}(B)$ and $\underset{i}{\cup} (Y(A_i)^{\varsigma})^{\star}(B)$ equate.  
\end{defn}

We thankfully have the following intuitive consitency:
\begin{prop} $\label{prop:asns}$
Every affine nc. scheme is a nc. scheme.  
\end{prop}
\begin{proof}
Let $X$ be an affien nc. scheme.  Since corepresentable functors are sheaves, so we only need verify the added clause but this is trivial since $\{ X \}$ is a familly equating on all $R$-algebras, and so prticularly on a select few of their images under some functor.  
\end{proof}
\break

\subsection{Theory over a commutative basering}

There are a very many reasons why to be interested in studying nc. schemes over commutative baserings rather than arbitrary baserings.  In short the theory is much richer and well behaved.  The following is actually the central geometric interest in this more well-behaved theory, lliking the ideas of essential smoothness above to concrete geometric interpretations.  

\begin{lem} $\label{lem:presubf}$
If $R$ is a commutative ring and if $A$ and $B$ are associative unital commutative $R$-algebras, then $Hom_{_RAlg}(A,B)$ is a proper subset of $Y(A)(B)$.  
\end{lem}
\begin{proof}
Every morphism of associative unital commutative $R$-algebras is by definition a morphism in $\mathfrak{_RAlg}$, however the zero morphism $A \rightarrow B$, mapping every element of $A$ to the additive identity of $B$ is a morphism in $\mathfrak{_RAlg}$ and not in $_RAlg$.  Since, the $0$-algebra is an object in $\mathfrak{_RAlg}$ and $_RAlg$, then $Hom_{_RAlg}(A,B)$ is always a proper subset of $Y(A)(B)$.  
\end{proof}

We're motivated to embed the category of (commutative) affine schemes into the category $[\mathfrak{_RAlg} : Sets]$ however, the chosen embedding will actually imply that these objects are not nc. affine schemes in of themselves, actually they will not even be general nc. schemes.  A priori, this seems to be a setback, however we will use this to our advantage to introduce a second type of building block and create alltogether new objects by gluing nc. schemes with \textit{"commutative"} schemes.  In memorial of their magnificently abstract colors and textures, we'll call these \textit{Pollock schemes}.


\subsection{A first algebraic adjunction}

\textit{From now on we assume our basering $R$ to be commutative also, (that is $R$ is commutative unital and associative).  }

Every commutative unital $R$-algebra $A$ has the underlaying structure of an associative $R$-algebra, likewise any unit preserving homomorphism of commutative unital $R$-algebras is itself a homomorphism of $R$-algebras.  Straigtway from this description we can observe that forgetting this extra structure preservers composition and if two morphisms differ, then they certainly differ when considering less structure.  In short, there is a \textit{forgetful functor} $\mathscr{F}$ from the category $_RAlg$ to the category $\mathfrak{_RAlg}$.  

Our first object is to construct its \textit{left} adjoint functor.  Infact we have already done so.  First we note that if $R$ is commutative then the inclusion $Z(R) \rightarrow R$ is an isomrophism therefore the functors $-^{\varsigma}$ and $-^{ab}\circ -^1$ are, as noted above naturally isomorphic.  That said, for the remainder of the text we identify $-^{\varsigma}$ with $-^{ab}\circ -^1$.

\begin{prop}
The standardization functor $-^{\varsigma}$ is left adjoint to the forgetful functor $\mathscr{F}$.  
\end{prop}

Instead of prooving the result in one shot, I prefer to break it up in some lemmas for organizational purpouses and for generality.

\begin{lem}
The unitization functor $-_1$ is left adjoint to the forgetful functor $\Psi: _RAlg\rightarrow _R\mathfrak{Alg}^1$.  
\end{lem}
\begin{proof}
Let $B$ be a commutative unital $R$-algebra and $A$ to be a unital $R$-algebra.  

We first note that every unit preserving morphism, from $ A_1\cong A \oplus R \rightarrow B$ can be decomposed into the direct of $\tilde{f}\oplus g$ two $R$-homomorphisms, with $g$ preserving the unit, this is by definitoin of abelianiztion being a direct sum or algebras.  However, since $g:B \rightarrow R$ is a $R$-homorphism preserving the unit, then it must be the identity on $R$ where $R$ is viewed as a sub-algebra of $B$.  Therefore $f = \tilde{f} \oplus 1_R$

Then the set functions: $\mathfrak{f}_{B,A}: Hom_{_R\mathfrak{Alg}^1}(A_1,B) \rightarrow Hom_{_RAlg}(A,\Psi(B))$ mapping a unit preserving $R$-homomorphism $f = \tilde{f} \oplus g: A_1 \rightarrow B$ to the $R$-algebra homomorphism $\tilde{f} \oplus 0$, where $0$ is the zero $R$-algebra homomorphism \textit{(obviously the later does not preserve the unit, and so only exists in the later category and not in the former)}.  Since $A\cong A \oplus 0$, where $0$ is the \textit{(non-unital)} zero algebra, then $\mathfrak{f}_{B,A}$'s codomain is indeed isomorphic to $Hom_{_RAlg}(A_1,B)$.  

Again we fall back on the identification of $A$ with $A \oplus 0$, and of a unit preservign $R$-algebra homomorphism as the direct sum desribed above.  Define the set functions $\mathfrak{g}_{B,A} :  Hom_{_RAlg}(A,\Psi(B)) \rightarrow Hom_{_R\mathfrak{Alg}^1}(A_1,B)$, taking an $R$-algebra homomorphism $g$ to the unit preserving $R$-algebra homomorphism $g \oplus 1_R$.  

Now, $(\mathfrak{f}_{B,A} \circ \mathfrak{g}_{B,A})(\tilde{f}\oplus 1_R)$ 

$= \mathfrak{f}_{B,A} \tilde{f}\oplus 0 = \mathfrak{f}_{B,A} \tilde{f} = \tilde{f} \oplus 1_R$.  

Conversly, $(\mathfrak{g}_{B,A})\circ \mathfrak{f}_{B,A} (f)$ 

$= (\mathfrak{g}_{B,A}) f\oplus 1_R = f \oplus 0 = f$.  

Therefore, for each pair $B$ and $A$ as above $\mathfrak{g}_{B,A}$ is the twosided inverse of $\mathfrak{f}_{B,A}$.  Finally, the naturality of these followws fro mth e naturaily of the direct sum.  

Therefore, $<-_1,\Psi, f_{B,A}>$ is an adjunction of functors.  
\end{proof}
\begin{lem}
The abelianization functor is left adjoint to the forgetful functor $\Phi$ from $\mathfrak{_RAlg}^1$ to its subcategory $_RAlg$.  
\end{lem}
\begin{proof}
This is essentially nothing but a restatement of the first isomorphisms theorem for algebras.  
\end{proof}
Two glue these two ideas together we can use this little category theoretical lemma.  
\begin{lem} $\label{lem:compos}$
$F_1 \adj G^1$ is an adjoint pair between of functors between two categories $\mathfrak{A}$ and $\mathfrak{B}$ and $F_2\adj G^2$ is also an adjoint pair between the categories $\mathfrak{B}$ and $\mathfrak{C}$, then $F_2\circ F_1\adj G^1\circ G^2$ is an adjoint pair between the categories $\mathfrak{A}$ and $\mathfrak{C}$.  
\end{lem}
\begin{proof}
Let $a$ be an object in $\mathfrak{A}$, $c$ be an object in $\mathfrak{C}$ and b ,$\tilde{b}$ be objects in $\mathfrak{B}$.  

The adjunctions hypothesized give:
$C(F_2(b),c)\cong B(b,G^2(c))$ and $B(F_1(a),\tilde{b}) \cong A(a,G^1(\tilde{b}))$.  
Now setting b:=$F_1(a)$ and $\tilde{b}:=G^2(c)$ it may be concluded:
$A(a,G^1(G^2(c))) \cong B(F_1(a),G^2(c)) \cong C(F_2(F_1(a)),c)$.  

So $F_2\circ F_1$ is left adjoint to $G^1\circ G^2$.  
\end{proof}

Now the proof of the proposition:
\begin{proof}
We simply take, the functors $F_1$ and $F_2$ to be the abelianization and unitization functors respectivly.  Also, we take the functors $G_1$ and $G_2$ in the above lemma to be the forgetful functors $\Phi$ and $\Psi$ respectivly.  

Now the composition of forgetful functors is again a forgetful functor and the composition $-^{ab}\circ -_1$ is by definition the standardization functor, if the base ring is commutative.  Therefore, the result follows by the last lemma.  
\end{proof}

\subsubsection{Another note as to why the basering needed to be commutative}
As a general note, the reason we had to assume the base ring was commutative is that the functor $\tilde{i}$, induced by the inclusion of $Z(R)$ into $R$ is a left adjoing functor, and generally has not right adjoint.  
In the case that $Z(R)\cong R$ is an isomrophism and so the functor $\tilde{i}:_{Z(R)}\mathfrak{Alg} \rightarrow _R\mathfrak{Alg}$ is an equivalence of categories, and so in particualr it has a right adjoint \textit{(namely the its inverse)} we have lemma $\autoref{lem:compos}$ and the uniquenes of adjoints imply the natural isomorphisms between of $^{ab} \circ ^1\circ \tilde{i}$ with $^{ab} \circ ^1$.

\subsection{A geometric adjunction}
As is usually the case in algebraic geometry we transform an algebraic result into a geometric one.  We \textit{"invert"} the two adjunct functors studied briefly above as notice that they still form an adjunction.  

Implicitly above we were using this crucially important functor, which may be defined in general, regarless if the basering $R$ is commutative or not.  

\begin{defn}
\textbf{Pollution Functor}

The functor $\mathfrak{P(-)}: [_{Z(R)}Alg:Sets] \rightarrow [\mathfrak{_RAlg}:Sets] $ taking a functor $P$ in $[RAlg:Sets]$ to the functor $P\circ -^{\varsigma}$ is called the \textbf{Pollution functor} and $\mathfrak{P}(P)$ is called the erradication of $P$.  
\end{defn}
\begin{lem}
The polution functor is indeed a functor. 
\end{lem}
\begin{proof}
This is a direct consequence of lemma $\autoref{lem:precomp}$.  
\end{proof}

Now er return to our setting of $R$ being commutative.  

\begin{defn}
\textbf{True Essence Functor}

The functor $\mathscr{T(-)}: [\mathfrak{_RAlg}:Sets] \rightarrow [_RAlg:Sets]$ taking a functor $P$ in $[\mathfrak{_RAlg}:Sets]$ to the functor $P\circ \Phi \circ \Psi$ is called the \textbf{true essence functor} and $\mathscr{T}(P)$ is called the \textbf{true essence} of $P$.  
\end{defn}

\begin{lem}
The true essence functor is a functor. 
\end{lem}
\begin{proof}
This is a direct consequence of lemma $\autoref{lem:precomp}$.  
\end{proof}

Before tackling the main result of this subsection we present a category theoretical lemma and a defnition that will streamline the proof:

\begin{lem}
If $F: \mathfrak{C} \rightarrow \mathfrak{D}$ is the left adjoint to a functor $G$.  Then the inverse functors $F^{\star}$ is right adjoint to the inverse functor $G^{\star}$.   
\end{lem}

\begin{proof}
Let $Q \in [\mathfrak{D}:Sets]$,$A \in \mathfrak{D}$ $P \in [\mathfrak{C}:Sets]$ and $B \in \mathfrak{C}$.  

If $F: \mathfrak{C} \rightarrow \mathfrak{D}$ is left adjoint to $G$, then there are unit and counit natural transformations $\eta: F \circ G \rightarrow 1_{\mathfrak{D}}$ and $\epsilon: G\circ F \rightarrow 1_{\mathfrak{C}}$ respectively.  

These induce natural ismorphisms $\epsilon^{\star}: G^{\star}\circ F^{\star} (Q)(A) = Q ((F\circ G) (A)) \rightarrow Q(1_{\mathfrak{D}}(A)) = Q(A) \cong 1_{[\mathfrak{D}:Sets]} (Q(A))$.  Therefore, the counit has induces a unit.  

Dually, we have $\eta^{\star}: F^{\star}\circ G^{\star} (P)(B) = P ((G\circ F) (B)) \rightarrow P(1_{\mathfrak{C}}(B)) = P(B) \cong 1_{[\mathfrak{C}:Sets]} (P(B))$ and similarly, the unit has induced a counit.

Therefore if $F$ is left adjoint to $G$, then $G^{\star}$ is left adjoint to $F^{\star}$.  
\end{proof}

\begin{thrm}
\textbf{First Adjunction Theorem} $\label{first adjunction theorem}$

The pollution functor $\mathfrak{P}$ is right adjoint to the true essence functor $\mathscr{T}$.  
\end{thrm}
\begin{proof}
By definition the true essence functor is just the inverse functor of the standardization functor and the pollution functor is the inverse image of the forgetful functor $\Psi \circ \Phi$.  

Since $-^{\varsigma} \adj (\Psi \circ \Phi)$, then the conclusion follows directly from the lemma.  
\end{proof}

\chapter{Approximations}

We will now completly abuse the existence of this adjunction of functors to approximate every nc. scheme both above and below by commutative ones.  

\section{Outer Approximations}
First we approximate our nc. schemes from the outside, by larger \textit{"schemes"} containing them.  

First we note, that \textit{all} our schemes were simply sheaves, on a certain site.  Therefore, there does exist sheafification functors $\Lambda$ from the categorie $[_R\mathfrak{Alg}:Sets]$ to the category of nc. schemes and $\Lambda_1$ from the categorie $[_R\mathfrak{Alg}:Sets]$ to the category of schemes.  

\begin{cor} $\label{cor:approx1}$
If $f:Z\rightarrow W$ is a morphism of nc. schemes, then there exists a unique morphism of schemes $\hat{f}:X\rightarrow Y$, and unique epis of nc. schemes $\epsilon_1:  \mathfrak{P }(X) \rightarrow Z$ and $\epsilon_0: \mathfrak{P }(Y) \rightarrow W$ completing to the following commutative diagram. 

$
\begin{tikzpicture}
  \matrix (m) [matrix of math nodes,row sep=3em,column sep=4em,minimum width=2em]
  {
     \mathfrak{P}(X) & Z \\
     \mathfrak{P}(Y) & W \\};
  \path[-stealth]
    (m-1-1) edge [dashed] node [left] {$\Lambda (\mathfrak{P}(\hat{f}))$} (m-2-1)
            edge [dashed,->>] node [below] {$\epsilon_0$} (m-1-2)
    (m-2-1) edge [dashed,->>] node [below] {$\epsilon_1$} (m-2-2)
    (m-1-2) edge node [right] {$ f $} (m-2-2);
\end{tikzpicture}
$

Dually, if $g:X\rightarrow Y$ is a morphism of schemes, then there exists a unique morphism of nc. schemes $\hat{g}:X\rightarrow Y$, and unique monics of schemes $\mu_0: \mathscr{T}(Z) \rightarrow X$ and $\mu_1: \mathscr{T}(W) \rightarrow Y$ completing to the following commutative diagram.  

$
\begin{tikzpicture}
  \matrix (m) [matrix of math nodes,row sep=3em,column sep=4em,minimum width=2em]
  {
     \mathscr{T}(Z) & Y \\
     \mathscr{T}(W) & X \\};
  \path[-stealth]
    (m-1-1) edge [dashed] node [left] {$\mathscr{T}(\hat{g})$} (m-2-1)
    (m-1-2) edge [dashed] node [below] {$\mu_0$} (m-1-1)
    (m-2-2) edge [dashed] node [below] {$\mu_1$} (m-2-1)
    (m-1-2) edge node [right] {$ g $} (m-2-2);
\end{tikzpicture}
$
\end{cor}
\begin{proof}
Since $\mathfrak{P}$ is left adjoint to the functor $\mathscr{T}$, then $M:=\mathscr{T} \circ \mathfrak{P}$ is a monadic functor and $C:= \mathfrak{P}\circ \mathscr{T}$ is a comonadic functor.  

Therefore, if $Z$ nc. scheme, \textit{Beck’s monadicity theorem} implies that Z is a quotient object of $M(Z)$, the univeral property of quotient objects gives the uniqueness of the epimorphism $\epsilon_0': M(Z)\rightarrow Z$ \textit{(and likewise for $\epsilon_1': M(W)\rightarrow W$)}.  

Now our choice of commutative schemes $X$ and $Y$ have already bee impicitly imposed by the constructed monad, that is $X$ must be $\mathscr{T}(Z)$ and $Y$ must be $\mathscr{T}(W)$. Likewise, $\hat{f}$ is forced to be the morphsim $\mathscr{T}(f)$.    

This entire construction is purely functorial, therefore everything is uniquely determined.  

The second claim follows likewise, by the \textit{duality principle}.  
\end{proof}

The construction above was entirely functorial, based of the (co)monad \textit{($\mathfrak{P}\circ \mathscr{T}$ and)} $ \mathscr{T} \circ \mathfrak{P}$.  This universal arrow's utmost importance merits a name, as it is one of the two central constructions of this entire paper.  

\begin{defn}
\textbf{Outer Universal}

For any morphism $f$ of nc. schemes, the universal morphism $\hat{f}$ described in the theorem above is called the \textbf{outer universal} of $f$.  
\end{defn}

The construction of the outer universal $\hat{f}$ of a morphism $f$, allows for the approximation of $f$ by a morphism of strictly larger \textit{(in the quotient object sence)} morphism nc. schemes $\Lambda(\mathfrak{P}(\hat{f}))$ which originates in a unique functorial manner from \textit{the} morphism of schemes $\hat{f}$.  Therefore, we may view $f$ as having properties similar to its corresponding outer universal $\hat{f}$.  

For example, if $\hat{f}$ is itself etale, then we may view $f$ as being \textit{"something that lives inside something which was etale"}.  So we first approximate $f$ from the outside as follows:

\begin{defn}
\textbf{Outer-Unramified}

A morphism $f$ of nc. schemes is said to be \textbf{Outer-Unramified} \textit{if and only if} its outer universal $\hat{f}$, is unramified.  

\end{defn}
Likewise:
\begin{defn}
\textbf{Outer-Smooth}

A morphism $f$ of nc. schemes is said to be \textbf{Outer-Smooth} \textit{if and only if} its outer universal $\hat{f}$, is smooth.  

\end{defn}
And similarly.  
\begin{defn}
\textbf{Outer-Etale}

A morphism $f$ of nc. schemes is said to be \textbf{Outer-Etale} \textit{if and only if} its outer universal $\hat{f}$, is etale.  

\end{defn}

\section{Inner-Approximations}
\subsection{Preliminaries}

We now estimate nc. schemes from the inside by embedding the category of affine schemes into $[_R\mathfrak{Alg}:Sets]$.  We now fuse these concepts with the very general ideas presented in the first part of this paper, to establish powerful notions of smoothness, unramifiedness and etaleness of a nc. scheme \textit{"from the within"} by means of \textit{"commutative"} affine schemes.

\subsection{A second adjunction theorem}
 
For the time being we restirct our attention to the category $_RAlg^1$ of unital $R$-algebras with non-unit preserving $R$-algebra homomorphisms.  This will permit us to get very accurate approximations of our nc. schemes.  However we'll need to fix some terminology so we may comunicate.  

\begin{defn}
\textbf{Unital nc. Scheme}

A \textbf{unital nc. scheme} is a nc. scheme $X$ such that there exists a familly $\{ U_i \}_i$ in $_R\mathfrak{Alg}^{1}$, with the property that for each $R$-algebra B, there are natural bjections on each $X(B) \leftrightarrow \underset{i}{\cup} Y(U_i)(B)$.  
\end{defn}
This is simply a more stringent restriction on the second condition in the dfinitoin of nc. scheme, asking that the \textit{"open cover of the commutative part of $X$"} can be covered by affine nc. schemes who's coordinate ring is a unital $R$-algebra.  

\begin{prop}
The category of unital nc. schemes is a full subcategory of the category of nc. schemes.  
\end{prop}
\begin{proof}
Directly from the defintiion of a nc. scheme as it did not impose any additional conditions on  morphism between unital nc. schemes.  
\end{proof}

\begin{prop}
There is an embeding of the category $(_R\mathfrak{Alg}^1)^{op}$ into the category of unital nc. affine schemes.  
\end{prop}
\begin{proof}
Consider the contravariant Yoneda embedding of $_R\mathfrak{Alg}^1$ into the category $[_R\mathfrak{Alg}:Sets]$, then $Y(A)$ must be a unital nc. scheme if $A$ was unital.  
\end{proof}

\begin{lem} $\label{lem:exten}$
There is a fully faithful functor from the category of schemes over $R$ to the category $[_RAlg:Sets]$.  Moreover, the evaluation of this functor on a commutative unital associative $R$-algebra gives the same pointset as the initial functor's evaluation thereon.  
\end{lem}
\begin{proof}
Let $X$ be an affine scheme over $R$, by the contravariant Yoneda embedding, we may identify $X$ with the functor $Hom_R(A,-):_RAlg \rightarrow Sets$ for some $R$-algebra $A$.  

Now consider the functor $-_{!}:[_RAlg:Sets] \rightarrow [_R\mathfrak{Alg}:Sets]$ defined as taking a morphism $f$ in $_RAlg$ and giving it the morphism:

$
P_!(f)(\mapsto 
\begin{cases} f & \mbox{ if $f$ is a morphism in $_RAlg$ \textit{(as viewed in $_R\mathfrak{Alg}$)} } \\
"\emptyset" & \mbox{ else } \end{cases}$, in $_R\mathfrak{Alg}$.  
Where by $"\emptyset"$ we mean a morphism is mapped to a morphism in the empty subcategory of $[_R\mathfrak{Alg}:Sets]$, with that technical poitn adressed, thre rest is a straigtforward and usual verfification ensure that $-_!$ is actually a functor.  
\end{proof}

Recallt that subfunctors $G,H$ of a functor $F$ in a category such as $[_R\mathfrak{Alg}:Sets]$ may be partially ordered with $G \leq H$ if and only if for every $R$-algebra $A$ $G(A)$ is a subset of $H(A)$.  With this topos-theroretic notion brought to the forefront of our mind we continue our discussion.  

\begin{thrm}
Every unital nc. scheme $X$ contains a largest subfunctor, in $[_R\mathfrak{Alg}:Set]$ of the form $Y_!$ for some scheme, moreover this subfunctor is unique up to natural isomrophism in $[_R\mathfrak{Alg}:Sets]$.  
\end{thrm}
\begin{proof}
We begin by the affine case.  If $Y(A)$ is a unital nc. scheme then its coordinate ring is unital and there is an epimorphisms $A \rightarrow A^{ab}$ with the universal property that any other morphism $A \rightarrow B$ must factor through it, making $A^{ab}$ the smallest quotient object of $A$.  Now by the duality of the contravariant Yoneda embedding, there $Y(A^{ab})$ must be the unique largest \textit{(with respect to the aforementioned partial order on subfunctors in $[_R\mathfrak{Alg}:Set]$)} subfunctor of $Y(A)$.  

Lemma $\autoref{lem:presubf}$ implies that for every commutative unital associative $R$-algebra B, $Hom_{_RAlg}(A^{ab},B)$ is the largest subset of $Y(A^{ab})(B)$ corresponding to an affine scheme.  By lemma $\autoref{lem:exten}$, we may embed the affine scheme $Hom_{_RAlg}(A^{ab},B)$ to the category $[_R\mathfrak{Alg}:Sets]$ in a way preserving the pointsets given by its evaluation on  every $R$-algebra.  Therefore, $Hom_{_RAlg}(A^{ab},-)_!$ is the unique smallest affine scheme \textit{(embeded as a subfunctor into $[_R\mathfrak{Alg}:Sets]$)} contained in $Y(A)$.  This takes care of the claim in the affine case.  

Now for the general case.  By definition, if $X$ is a unital nc. scheme then there is a familly of affine nc. schemes $Y(U_i)$ with dual to some unital associative $R$-algebras $U_i$, such that the functor $\underset{i}{\cup} Y(U_i)$ is nautrally isomorhpic to $X$.  Therefore $\underset{i}{\cup} Y(U_i)$ is the largest subfunctor of $X$, and each $Y(U_i)$ are corepresetnable subfunctors of $X$.  By the above discussion there are unique maximal subfunctors $Hom_R(U_i^{ab},-)_!$ corresponding to each $Y(U_i)$.  Therefore, $\underset{i}{\cup} Hom_R(U_i^{ab},-)_!$ must be the largest \textit{(up to natural isomorphism)} subfunctor of $X$.  If $\underset{i}{\cup} Hom_R(U_i^{ab},-)$ is infact a scheme, then we are done, if not then we get the smalles possible scheme containing $\underset{i}{\cup} Hom_R(U_i^{ab},-)$ by applying the sheafification functor $\Lambda_1$ to it and so $\Lambda_1 (\underset{i}{\cup} Hom_R(U_i^{ab},-)_!$ is the unique \textit{(up to natural isomrophism)} largest subfunctor of $X$ which originates from a sheaf.  
\end{proof}

Before proving the last main result, we'll need a little lemma.  
\begin{lem} $\label{lem:split}$  
If $F$ is a subfunctor of $G$ in $[_R\mathfrak{Alg}:Sets]$ then there is a split monomorphism $\mu : F \rightarrow G$.  
\end{lem}
\begin{proof}
If $F$ is a subfunctor of $G$, then for every $R$-algebra $A$, there is a natural familly of inclusions $i_A:F(A)\rightarrow G(A)$.  Since inclsuions are left invertible, this gives rise to a natural familly of morphisms $j_A$ each with the property that $j_A\circ i_A = 1_{F(A)}$.  
\end{proof}

\begin{cor} 
\textbf{Inner Approximation Theorem}

If $f:Z\rightarrow W$ is a morphism of nc. schemes, then there exists a unique morphism of schemes $g:X\rightarrow Y$, and unique split monic natural transformations $\mu_0: X_! \rightarrow Z$ and $\mu_1: Y_! \rightarrow W$ completing to the following commutative diagram.  

$
\begin{tikzpicture}
  \matrix (m) [matrix of math nodes,row sep=3em,column sep=4em,minimum width=2em]
  {
     X_! & Z \\
     Y_! & W \\};
  \path[-stealth]
    (m-1-1) edge [dashed] node [left] {$g_!$} (m-2-1)
    (m-1-1) edge [dashed] node [below] {$\mu_0$} (m-1-2)
    (m-2-1) edge [dashed] node [below] {$\mu_1$} (m-2-2)
    (m-1-2) edge node [right] {$ f $} (m-2-2);
\end{tikzpicture}
$
\end{cor}
\begin{proof}
The monic natural transformations $\mu_0$ and $\mu_1$ are exactly the unique ones describing the unique maximal subfunctor $X_!$ of $Z$ and $Y_!$ of $W$ arising as the images of an affine scheme embeded into $[_R\mathfrak{Alg}:Sets]$ under the functor $_!$.  Now, as noted above in lemma $\autoref{lem:split}$ these monomorphisms $\mu_0$ and $\mu_1$ must be split monos.  

Therefore, in particular $\mu_1$ has a left inverse and so consider the morphism: $\mu_1^{1-}\circ f\circ \mu_0: X_! \rightarrow Y_!$ making the diagram commute.  However, since the functor $-_!$ embeds the category $[_RAlg:Sets]$ into $[_R\mathfrak{Alg}:Sets]$, the morphism
$\mu_1^{1-}\circ f\circ \mu_0$ must be the unique $-_!$-image of a morphism $g$ between the functors $X$ and $Y$.  Now since $X$ and $Y$ are schemes and the category of schemes is a full subcategory of the $[_RAlg:Sets]$ then $g$ must be a morphism of schemes, \textit{(uniquely making the diagram in question commute)}.  
\end{proof}

Since the universal morphism $g$ in the above cosntruction was uniquely determined by $f$, given by a functorial construction, we give that functor a name.  

\begin{defn}
\textbf{Inner Universal}

For any morphism $f$ of nc. schemes, its \textbf{Inner Universal} is the \textit{unique} morphism $\bar{f}$, denoted by $g$ in the \textit{Inner Approximation Theorem}.  
\end{defn}
Any universal property is given by a functor and in particualr the functor $\bar{-}: f \mapsto \bar{f}$ is such a construct.  Since $\bar{f}$ is the uniquely determined morphism essentially coinciding with $f$ on the largest subfunctors of $f$'s source and target, which both arise from \textit{(classica)} schemes we approximate $f$'s properties by $\bar{f}$'s properties \textit{from the inside)}.  

\subsubsection{Inner Approximations}

\begin{defn}
\textbf{Inner-Unramified}

A morphism $f$ of nc. schemes is said to be \textbf{Inner-Unramified} \textit{if and only if} its inner universal $\bar{f}$ is unramified.  

\end{defn}
Likewise:
\begin{defn}
\textbf{Inner-smooth}

A morphism $f$ of nc. schemes is said to be \textbf{Inner-Smooth} \textit{if and only if} its inner universal $\bar{f}$ is smooth.  

\end{defn}
And similarly.  
\begin{defn}
\textbf{Inner-Etale}

A morphism $f$ of nc. schemes is said to be \textbf{Inner-Etale} \textit{if and only if} its inner universal $\bar{f}$ is etale.  

\end{defn}

\section{Total-approximations}
We now combine the ideas of inner and outer approximations to describe a very accurate way of defining the concepts of smoothness, unramifiedness and etalness of morphisms by understanding them both from the \textit{unique} largest morphism of schemes the "contain" and the smallest \textit{unique} morphism of schemes "containing" them.  

\begin{defn}
\textbf{Nc. Unramified Morphism}

A morphism $f$ of nc. schemes is said to be \textbf{Nc. Unramified} \textit{if and only if} it is both \textit{inner-unramified} and \textit{outer-unramified}.  

\end{defn}
Likewise:
\begin{defn}
\textbf{Nc. Smooth Morphism}

A morphism $f$ of nc. schemes is said to be \textbf{Nc. Unramified} \textit{if and only if} it is both \textit{inner-smooth} and \textit{outer-smooth}.  

\end{defn}
And similarly.  
\begin{defn}
\textbf{Nc. Etale Morphism}

A morphism $f$ of nc. schemes is said to be \textbf{Nc. Etale} \textit{if and only if} it is both \textit{inner-etale} and \textit{outer-etale}.  

\end{defn}

\subsubsection{Simple characerisations of Nc. Smoothness}
We now characerise our main interest, which was the defintion of a geometrically intuitive and appropriate notion of smoothness over nc. schemes.  

\begin{thrm}
The following are equivalent:

- $f:W \rightarrow Z$ is a nc. smooth morphism

- Both the sheaves $\Omega_{\bar{Z}|\bar{W}}$ and $\Omega_{\hat{Z}|\hat{W}}$ are locally free

 \textit{(where $\hat{f}:\hat{W}\rightarrow \hat{Z}$ is the outer universal of $f$ and $\bar{f}:\bar{W}\rightarrow \bar{Z}$ is its inner universal)}.

\end{thrm}
\begin{proof}
By definitoin of nc. smoothness, $f$ is nc. smooth if and only if both $\hat{f}$ and $\bar{f}$ are smooth morphisms of schemes.  Since a morphism of schemes is smooth if and only if its correesponding sheaves of differentials if locally free, we conclude our case.  
\end{proof}

\section{The Consistency Theorem and Examples}
We now conclude with this behemoth, showing that any commutative scheme when viewed as a nc. scheme via the functor $-_!$, enjoys all the same properties it did when viewed simply as a scheme.  In this sense, unital non-commuttive schemes over a commtuative unital associative basering, strictly generalize the properties of their classical counterparts.  

\begin{thrm}
\textbf{Consitency Theorem}
If $f$ is a morphism of commutative schemes, then $\Lambda(f_!)$ is :

- nc. smooth if and only $f$ is smooth.

- nc. unramified if and only if $f$ is unramified.

- nc. etale if and only if $f$ is etale.  
\end{thrm}
\begin{proof}
For any morphism of nc. schemes was seen to giveway to a unique commutative diragram:

$
\begin{tikzpicture}
  \matrix (m) [matrix of math nodes,row sep=3em,column sep=4em,minimum width=2em]
  {
     X_! & Z & \mathfrak{P}(X) \\
     Y_! & W & \mathfrak{P}(Y) \\};
  \path[-stealth]
    (m-1-1) edge [dashed] node [left] {$\bar{f}_!$} (m-2-1)
    (m-1-1) edge [dashed] node [below] {$\mu_0$} (m-1-2)
    (m-2-1) edge [dashed] node [below] {$\mu_1$} (m-2-2)
    (m-1-2) edge node [right] {$ f $} (m-2-2)
    (m-1-3) edge [dashed] node [right] {$\mathfrak{P}(\hat{f})$} (m-2-3)
            edge [dashed,->>] node [below] {$\epsilon_0$} (m-1-2)
    (m-2-3) edge [dashed,->>] node [below] {$\epsilon_1$} (m-2-2)
    (m-1-2) edge node [right] {$ f $} (m-2-2);
\end{tikzpicture}
$ 
with $\epsilon_0$ and $\epsilon_1$ epis, $\mu_0$ and $\mu_1$ monics and $g$ and $f$ morphism of commutative schemes.  

Sheafification is exact, therefore applying the functor $\Lambda\circ -_!$ to this diagram with the middle arrow particularised to our interest yeilds:
$
\begin{tikzpicture}
  \matrix (m) [matrix of math nodes,row sep=3em,column sep=4em,minimum width=2em]
  {
     \Lambda(X_!) & \Lambda(A_!) & \Lambda (\mathfrak{P}(X)) \\
     \Lambda(Y_!) & \Lambda(B_!) & \Lambda (\mathfrak{P}(Y)) \\};
  \path[-stealth]
    (m-1-1) edge [dashed] node [left] {$\bar{\Lambda(f)}_!$} (m-2-1)
    (m-1-1) edge [dashed] node [below] {$\Lambda(\mu_0)$} (m-1-2)
    (m-2-1) edge [dashed] node [below] {$\Lambda(\mu_1)$} (m-2-2)
    (m-1-2) edge node [right] {$ \Lambda(f_!) $} (m-2-2)
    (m-1-3) edge [dashed] node [right] {$\Lambda (\mathfrak{P}(\hat{f}))$} (m-2-3)
            edge [dashed,->>] node [below] {$\Lambda(\epsilon_0)$} (m-1-2)
    (m-2-3) edge [dashed,->>] node [below] {$\Lambda(\epsilon_1)$} (m-2-2);
\end{tikzpicture}
$ 

Now observe that the split monics $\Lambda(\mu_0)$ and $\Lambda(\mu_1)$ making the first square commtuate are unique.  So the monics $i: A_! \rightarrow \Lambda(A_!)$ and $i: B_! \rightarrow \Lambda(B_!)$ arising from the sheafification of the functors $A_!$ and $B_!$, together with the morphism $\bar{\Lambda(f)}:=f$ make the first square in the diagram commute.  Therefore by the uniqueness of that diagram we have that $\bar{\Lambda(f)}$ must be exactly what we tested it out to be, tht is simply $f$.  

Now we argule similarly concerning the second square.  Consider the morphism $\mathscr{T}(f)$ as being $\hat{f}$.  Then taking the morphisms $\epsilon_0$ and $\epsilon_1$ to be the $\Lambda$-images of the canonical epimorphism $\mathfrak{P}\circ\mathscr{T}(A_!)\rightarrow A_!$ arising from the monadic functor $\mathfrak{P}\circ\mathscr{T}$ we hav a commutative diagram on the left also.  Therefore, $\hat{f}$ must be describes by the morphism $\mathscr{T}(f)$.  However, $\mathscr{T}(f)$ is just the forgetful functor preapplied to the morphism $f$, so $\mathscr{T}(f)$ can be viewed as simply the morphism itself $f$ put back inside $[_RAlg:Sets]$. Therefore, in this case $\hat{\Lambda(f_!)}$ is just $f$, which is just $\bar{\Lambda(f_!)}$.  

Therefore, if $f$ is smooth \textit{(resp. unramified, resp. etale)} if and only if $\hat{\Lambda(f_!)}$ and $\bar{\Lambda(f_!)}$ are smooth textit{(resp. unramified, resp. etale)}.  Which by definition of nc. smothness \textit{(resp. nc. unramifiedness, resp. nc. etaleness)} is the case if and only if $f$ is is nc. smothness \textit{(resp. nc. unramifiedness, resp. nc. etaleness)}.  
\end{proof}
So, these concepts are a strict generalisation of the classical ones.  Using this geometrically intruitive defintion, we not only have that all commutative schemes moaintain their said properties when transfered into the non-commutative category, but this powerfull approach gives, tremendous universal proerties for alot our computational needs.  

\subsubsection{A Very Important Observation}
Notice that for example the conecpt of nc. smoothness in the proof above did not use any of the specific properties of smoothness, except that its best commutative approximations, that is both its outer and inner universal had this property.  Moreover, we're motivated to adopted this idea of approximating the property of smoothness in the non-commutative categry, since not do classical scheme maintina this property when viewed as nc. schemes via the functor $\Lambda(-_!)$ which preserves pointsets up to sheafification, but many objects which should be be smooth become nc. smooth, such as $Y(R<X,Y>)$.  

This promts the following generealised version of the above theorem, which should provide enough motivation for reconsidering how to view non-commutative geometry.  In the name of accuracy, we first need to be logicians for a moment.  

\begin{defn}
\textbf{$\mathfrak{C}$-independent property}

Let $\mathfrak{C}:= <Obj_{\mathfrak{C}},Mor_{\mathfrak{C}},Source_{\mathfrak{C}},Target_{\mathfrak{C}},\circ_{\mathfrak{C}} >$ be a category.  

A property \underline{\textbf{P}} is said to be $\mathfrak{C}$-independent if it makes no reference to the any maps $Source_{\mathfrak{C}}, Target_{\mathfrak{C}}: Obj_{\mathfrak{C}} \rightarrow Mor_{\mathfrak{C}}$ and the map $\circ_{\mathfrak{C}} : Mor_{\mathfrak{C}} \times Mor_{\mathfrak{C}} \rightarrow \{ ture, false \}$; where $Source_{\mathfrak{C}}$ is the map assigning a morphism in ${\mathfrak{C}}$ the object which is its source,  $target_{\mathfrak{C}}$ is the map assigning a morphism in $\mathfrak{C}$ the object which is its target and $\circ_{\mathfrak{C}}$ assigning to a pair of morphisms  in $\mathfrak{C}$ a truth value \textit{true} if they are composable in $\mathfrak{C}$ or \textit{false} is they are not.  

Moreover, the property \underline{\textbf{P}} must make no reference to the classes $Obj_{\mathfrak{C}}$ and $Mor_{\mathfrak{C}}$ of objects and morphisms of $\mathfrak{C}$, respectivly.  
\end{defn}
We will be interested in properties that make no references to any category.  
\begin{defn}
\textbf{Category independent property}

A property \underline{\textbf{P}} is said to be category-independent if it is $\mathfrak{C}$-independent for \textit{every} catgory $\mathfrak{C}$.  
\end{defn}

\begin{thrm} $\label{MainTheorem}$
\textbf{The Nc. Approximation Theorem}
If $f$ is a morphism of commutative schemes and \underline{\textbf{P}} is a category independent property of that morphism, then both the inner and outer universals of $\Lambda(f_!)$ have the category independent property \underline{\textbf{P}}.  
\end{thrm}
\begin{proof}
The proof is a complete rephrasing of the above argument.  

For any morphism of nc. schemes was seen to giveway to a unique commutative diragram:

$
\begin{tikzpicture}
  \matrix (m) [matrix of math nodes,row sep=3em,column sep=4em,minimum width=2em]
  {
     X_! & Z & \Lambda (\mathfrak{P}(X)) \\
     Y_! & W & \Lambda (\mathfrak{P}(Y)) \\};
  \path[-stealth]
    (m-1-1) edge [dashed] node [left] {$\bar{f}_!$} (m-2-1)
    (m-1-1) edge [dashed] node [below] {$\mu_0$} (m-1-2)
    (m-2-1) edge [dashed] node [below] {$\mu_1$} (m-2-2)
    (m-1-2) edge node [right] {$ f $} (m-2-2)
    (m-1-3) edge [dashed] node [right] {$\Lambda (\mathfrak{P}(\hat{f}))$} (m-2-3)
            edge [dashed,->>] node [below] {$\epsilon_0$} (m-1-2)
    (m-2-3) edge [dashed,->>] node [below] {$\epsilon_1$} (m-2-2)
    (m-1-2) edge node [right] {$ f $} (m-2-2);
\end{tikzpicture}
$ 
with $\epsilon_0$ and $\epsilon_1$ epis, $\mu_0$ and $\mu_1$ monics and $g$ and $f$ morphism of commutative schemes.  

We particularise this diagram to our setting, as:

$
\begin{tikzpicture}
  \matrix (m) [matrix of math nodes,row sep=3em,column sep=4em,minimum width=2em]
  {
     X_! & \Lambda(A_!) & \Lambda (\mathfrak{P}(X)) \\
     Y_! & \Lambda(B_!) & \Lambda (\mathfrak{P}(Y)) \\};
  \path[-stealth]
    (m-1-1) edge [dashed] node [left] {$\bar{\Lambda(f)}_!$} (m-2-1)
    (m-1-1) edge [dashed] node [below] {$\mu_0$} (m-1-2)
    (m-2-1) edge [dashed] node [below] {$\mu_1$} (m-2-2)
    (m-1-2) edge node [right] {$ \Lambda(f_!) $} (m-2-2)
    (m-1-3) edge [dashed] node [right] {$\Lambda (\mathfrak{P}(\hat{f}))$} (m-2-3)
            edge [dashed,->>] node [below] {$\epsilon_0$} (m-1-2)
    (m-2-3) edge [dashed,->>] node [below] {$\epsilon_1$} (m-2-2);
\end{tikzpicture}
$ 

Now observe that the split monics $\mu_0$ and $\mu_1$ making the first square commtuate are unique.  So the monics $i: A_! \rightarrow \Lambda(A_!)$ and $i: B_! \rightarrow \Lambda(B_!)$ arising from the sheafification of the functors $A_!$ and $B_!$, together with the morphism $\bar{\Lambda(f)}:=f$ make the first square in the diagram commute.  Therefore by the uniqueness of that diagram we have that $\bar{\Lambda(f)}$ must be exactly what we tested it out to be, tht is simply $f$.  

Now we argule similarly concerning the second square.  Consider the morphism $\mathscr{T}(f)$ as being $\hat{f}$.  Then taking the morphisms $\epsilon_0$ and $\epsilon_1$ to be the $\Lambda$-images of the canonical epimorphism $\mathfrak{P}\circ\mathscr{T}(A_!)\rightarrow A_!$ arising from the monadic functor $\mathfrak{P}\circ\mathscr{T}$ we hav a commutative diagram on the left also.  Therefore, $\hat{f}$ must be describes by the morphism $\mathscr{T}(f)$.  However, $\mathscr{T}(f)$ is just the forgetful functor preapplied to the morphism $f$, so $\mathscr{T}(f)$ can be viewed as simply the morphism itself $f$ put back inside $[_RAlg:Sets]$. Therefore, $\hat{\Lambda(f_!)}$ is just $f$, that is $f=\hat{\Lambda(f_!)}$ which is just $\bar{\Lambda(f_!)}$.  

We therefore have the equalities $f = \hat{\Lambda(f_!)} = \bar{\Lambda(f_!)}$.  Therefore, $f$ has a category independent property \underline{\textbf{P}} \textit{if and only if} both $\hat{\Lambda(f_!)}$ and $\bar{\Lambda(f_!)}$ have the category independent property \underline{\textbf{P}}.  
\end{proof}

So the inner and outer universals of a morhpism of schemes's embedding into the category of nc. scheme share \textit{every single} one of the properties of that morphism.  Moreover, in general since every single morphism of nc. schemes has \textit{unique} inner and outer universals then approximating it from the inside and from the outside it would be utterly reasonable to \textit{extend} a defintiion of a category independent property \underline{\textbf{P}} of a morphism of schemes to the nc. category by stating that a morphism of nc. schemes has a certain category independent property \underline{\textbf{P}} if and only if its inner and outer universals have that category independent property in the category of schemes.  

By now the following should be utterly natural.
\begin{defn}
\textbf{Nc. morphism property}

If \underline{\textbf{P}} is a category independent property of a morphism $f$ of schemes, then we say a morphism $g$ of nc. schemes has the propery \textbf{Nc. \underline{\textbf{P}}} \textit{if and only if} both :

-the inner universal $\bar{g}$ of $g$ has the category independent property \underline{\textbf{P}}

and

-the outer universal $\hat{g}$ of $g$ has the category independent property \underline{\textbf{P}}.  

Any such property "\textbf{Nc. \underline{\textbf{P}}}", is called a \textbf{nc. morphism property.}  
\end{defn}

We think of an nc. morphism property as an approximation of a category independent property property of its analogous property of morphisms of schemes, set within the category of nc. schemes.  

Now any example you like follows directly, lets briefly peek at a nice one.  
\begin{ex}
A morphism $f$ of nc. schemes then, $f$ is nc. proper \textit{if and only if} both $\hat{f}$ and $\bar{f}$ are proper morphisms of schemes.  
\end{ex}

\part{Abstractions}
\chapter{Philosophy}

The idea of this section is not to introduce a concept that should be worked on in years to come in of itself.  Rather, I am of the opinion that the sole purpous of abstraction, is to truely understand the fundamental essence of the problem at hand; that is, to truely scritinise the underlaying machinery which play an actual role in what is happeneing.  Anything else can be therefore regarded as conceptual noise, though interest it is irrelevant to the actual issue being delt with.  

In this spirit, we present an abstraction of the problem of approximating a nc. scheme by a commutative one; simply to understand what is truely happening, and not to suggest that this vast abstraction of the issue's essentials should ever be studied or applied in any way other than philosophically.  

Now in the construction of a nc. scheme we noticed that nc. schems over $R$, built from unital $R$-algebras behaved particularly well, in that they permitted for the outer approximations to be possible.  The necessary property here is that the $R$-algebras by which they were constructed were actually monoids in the monoidal category of $R$-modules.  

Now, the fact that they were built from objects in an abelian category $\mathfrak{R}$, though nice for many homological applications is in itself unncessary for this construction.  Constrastingly their monoidal property and the monoidal property of $_RMod$ is totally fundamental.  Philospophically, a monoid is the natural abstration of an unital associative algebra, and so if were interested in building \textit{"scheme-like"} objects in a meaningful way, it would make sence that the story begin with objects that are conceptually representative of algebras.  Therefore the first step, is to consider the category of monoids arising from a monoidal category $\mathfrak{R}$ and their monoid morphisms.  Nostaligically, we denote this category $\mathfrak{_RAlg}$.  

The second step in our construction of schemes over $R$, was our pulling back of the Zariski topology on $_RAlg^{op}$ to $_R\mathfrak{Alg}^{op}$.  Though it was conceptually helpful, we did not ever actually use the fact that the baserings of $_RAlg^{op}$ and of $_R\mathfrak{Alg}^{op}$ were one and the same.  We very well could have used $R$'s center at the loss of some initial intuition; all we truely needed was the existence of a functor $L_1$ from $_RAlg$ to $_R\mathfrak{Alg}$ which brought the $R_1$-image of objects in $_R\mathfrak{Alg}$ to some corresponding, minimal quotient object under the natural partial ordering on objects in $_R\mathfrak{Alg}$.  For this to happen, it was sufficient that $L_1$ be a forgetful functor, that is $L_1$ be faithful and $R_1$ be its left adjoint; that is $R_1$ be a free functor.  However, what we did ask of the basering of $_RAlg$, was that it was commutative unital and associative, so that we may travel to and from our familiar classical setting, via our functors $L_1$, and $R_1$.  Infact, we notice that when we were approximating representable schemes in a minimal way from outside the induced adjunction between the inverse of the forgetful functor and the inverse of the standardization functor had become necessary.  

Therefore, we ask that there is an adjunction $L_1 \adj R_1$ of functors, with $L_1: _RAlg \rightarrow \mathfrak{_RAlg}$ being faithful.  Moreover, we necessitate that the basering $R$ of the category of $R$-algebras which we draw our approximations of the localization maps and therefore the Zarisky coverings on $\mathfrak{_RAlg}$ be a commutative unital associative ring.  

The third step was to then to look at the category of scheaves on our \textit{"Zarisky-like"} site $_R\mathfrak{A}^{op}$ and then identify our \textit{nc. schemes} as being sheaves which could be covered by nc. analogues to affine schemes over $R$, these were the representable sheaves.  Implicitly we were using a very important trait of $\mathfrak{_RAlg}$, that is that the symbolic bifunctor $Hom_{_R\mathfrak{Alg}}(-,-)$ took values in the category $Sets$.  In short if $Hom_{_R\mathfrak{Alg}}(-,-)$ did not do so, then there would be no such concept of a representable functor in $[_RAlg:Sets]$.  Therefore we implicitly required that the category $\mathfrak{_RAlg}$ be enriched over $Sets$; classically put, we need that $\mathfrak{_RAlg}$ be \textit{locally small}.  

However we may ask something less much weaker.  Our objective was to be able to compare affine schemes with these representable functors, so it would be enough to ask, that these object everntually meet up in the category of sets in a functorial manner.  In other words it seems to me the weakest thing we can ask for that is representative of the issue at hand is that the category $\mathfrak{_RAlg}$ be enriched over some concrete category $<\mathfrak{C},\Phi>$.  Therefore, for any object $A$ in $\mathfrak{_RAlg}$, we may compare the underlaying pointsets of the functor $\Phi(Hom_{\mathfrak{_RAlg}}(A,-))$ with the affine scheme $Hom_{_RAlg}(G(A),-)$ in a meaningful manner, where $\Phi:\mathfrak{C}\rightarrow Sets$ is the forgetful functor describing the concretness of $\mathfrak{C}$.  Moreover, since every forgetful functor from a concrete category to Sets has a left adjoint, $\Psi$ we may always push bring affine schemes back forom the category $[_RAlg:Sets]$ into the category $[\mathfrak{_RAlg}:\mathfrak{C}]$ for comparison via the functor $-_!\circ \Psi$.  So we need to atleast require that the category $\mathfrak{_RAlg}$ be \textit{enriched over some concrete category} $\mathfrak{C}$.  

Now in the fourthstep we provided approximations of our objects via from the outside and noted the $-_!\circ \Psi$-image of schemes over $R$.  Most importantly, noting that the sheafification $\Lambda$ of the $-_!\circ \Psi$ of an scheme over $R$, had all the same category independent properties as it did before the application of the functor $\Lambda \circ -_!\circ \Psi$ to it.  For this to be possible we need the existence of a sheafification functor from the category $[\mathfrak{_RAlg}: \mathfrak{C}]$ to the category of $\mathfrak{C}$-sheaves over the site $\mathfrak{_RAlg}^{op}$; that is we need the category $[\mathfrak{_RAlg}: \mathfrak{C}]$ to satisfy \textit{WISC}.  

Finally, we wanted appriximations of our nc. schemes from the outside by schemes over $R$.  We constructed the inner univeersal implicity by means of the adjuction $i_! \adj i^{-1}$.  In other words, the crutial ingredient was that there exist adjoint functors $L_2 \adj R_2$ with $L_2: [_RAlg:Sets] \rightarrow [\mathfrak{_RAlg}: \mathfrak{C}]$.  Infact we may notice, $L_2$'s source to actually even be $[_RAlg:Sets]$, strictly speaking it is enough for the functor $L_2$'s source to be the category of schemes $Csh_{Spec(R)}$, over the affine scheme $Spec(R)$, where $R$ is as above.  

Therefore, we have our before last necessary assumption, on the meaningful machinery here necessary, that is we require the existence of an \textit{adjunction} $L_2 \adj R_2$, with $L_1$'s source being a category of schemes over a scheme $X$ and its target is the category $[\mathfrak{_RAlg}: \mathfrak{C}]$.  

Now our final requirement for an abstration of a morphism of schemes to be meaningful and approximable from the ourside and from the inside by morphisms of schemes is one thing not mentioned in this paper.  A central underlaying tool, for example in Hochschild (co)homology is that the category of $\mathfrak{R}$-modules is abelian and has enough cyclic objects.  Therefore, for our abstraction of these concepts to be representative we need that the Hochschild and therefore cyclic and other module theoretic tools exist.  So we impose one last thing, we ask that the category $\mathfrak{R}$ on which we built all of our theory, is abelian and has cyclic objects.  That is, our final requirement is that for every object $\mathfrak{RAlg}$ there exist, a faithful cyclic-object preserving functor $C_Z$ from the category $\mathfrak{_RAlg}$ to the category $_ZMod_Z$ of $Z$-bimodules.  

It is easy to see that since the category $_{\mathfrak{R}}\mathfrak{Alg}$ is monoidal, then it must also contain cyclic objects and so our central tools in non-commutative geometry, the Hochschild (co)homology theories transfer over nicely.  

\break

\section{Schemes in General in a representative and meaningful Way}
We now resummarise the entire paper in abstract form.  

Our first building block is a category which behaves in the most crutial ways the category of unital associative algebras over a ring $R$.  
\begin{defn}
\textbf{$\mathfrak{R}$-algebras}
If $\mathfrak{R}$ is a $\mathfrak{C}$-enriched monoidal Grothendieck Category, with tensor product $\otimes_{\mathfrak{R}}$ and $\mathfrak{C}$ is concrete then:

The category of monoids and monoidal morphisms over $\mathfrak{R}$ is called the \textbf{category of $\mathfrak{R}$-algebras} and is denoted $\mathfrak{_RAlg}$.  
\end{defn}

Since our goal is to built accurate approximations of \textit{"scheme-like object"}, we'll need sometinhg to approximate from, so we draw from a category of $R$-algebras.  For this approximation to be in any way accurate, we'll need a very \textit{"weak equivalence"} of categories our $\mathfrak{R}$-algebras and the $R$-algebras over some commutative unital associative ring.  

\begin{defn}
\textbf{$R$-Standardization}

An functor $\varsigma: \mathfrak{_RAlg} \rightarrow _RAlg$, possesing a right adjoint, where $R$ is a commutative unital associative ring, is called an \textbf{$R$-standardization}.  
\end{defn}

Now therefore, we have a way of approximating the dual of an $\mathfrak{R}$ algebra in the largest possible way with respect ot $\varsigma$-from inside.  
\begin{prop}
For every $X$ in $\mathfrak{_RAlg}^{op}$, the object $^{op}\circ \varsigma (X^{op})$'s 

$^{op}\circ G \circ -^{op}$-image is the largest subobject of $X^{op}$, which is the $^{op}\circ G \circ -^{op}$-image of an affine scheme over $R$.  
\end{prop}
\begin{proof}
We identify the categories of affine schemes over $R$ and $_RAlg^{op}$.  

Since, $\varsigma$ has a right adjoint, which we call $G$, then functor $G \circ \varsigma$ is monadic.  Therefore for any object $X$, the universal property of the monad $G \circ \varsigma (X)$ is that if $Y$ is an object in $R_Alg$, then $G(Y)$ is a quotient object of $G\circ \varsigma (X)$.  That is, $\varsigma (X)$ is the $R$-algebra which induces the smallest quotient object under $X$ after applying $G$.  

By duality we have, $(\varsigma (X))^{op}$, yields the largest possible subobject of $X^{op}$ after applying the functor $^{op}\circ G \circ -^{op}$.  
\end{proof}

The category of affine schemes over $Spec(R)$ is dual to the category of $R$-algebras, we view the dual of $\mathfrak{_RAlg}$ as the category of \textit{"something like affine schemes over something like Spec of $\mathfrak{R}$"}.  

In order to glue these \textit{"affine schemes over something like Spec($\mathfrak{R}$)"} together we'll need a topology on $\mathfrak{_RAlg}^{op}$.  We want the outcome of this glueing to be scheme-like objects, in the closest possible sence, so we'll approximate the Zarisky site as closely as is possible.  We now make use of this weak equivalence; that is of this adjunction to approximate a Zarsiky topology on $\mathfrak{_RAlg}$ which is \textit{exactly} the Zarisky site on \textit{"the largest affine part with respect to $^{op}\circ\varsigma\circ ^{op}$"} of our \textit{"scheme-like objects"}.  

\begin{defn}
\textbf{$\varsigma$-relative Zariski Basis on $\mathfrak{_RAlg}$}

If the category $\mathfrak{_RAlg}$ possese an $R$-standardization, then we define a basis for a covering $Cov(\mathfrak{_RAlg}^{op})$ on $\mathfrak{_RAlg}^{op}$ as follows:

A small familly of morphisms $\{ \phi_i: A_i \rightarrow X \}$ is a covering on $X$ if and only if:

- \underline{\textbf{$\varsigma$-relative Zariski-open immersions axiom:}} 

For ever index $i$ there exist elements $r_i$ in $\varsigma(X^{op})$ and there exist isomorpihsms $\psi_i: \varsigma(A^{op}) \rightarrow (\varsigma(X^{op}))_{r_i}$ such that the composition $\psi_i \circ \varsigma( \phi_i^{op}) : \varsigma(X^{op}) \rightarrow (\varsigma(X^{op}))_{r_i}$ is exactly the localization map of $\varsigma(X^{op})$ at $r_i$.  

- \underline{\textbf{$\varsigma$-relative Zariski-partition of unity axiom:}} 

For every index $i$, there exist elements $f_i$ in $\varsigma(X^{op})$, such that:

$\underset{i}{\sum} f_ir_i =1$.  

Consider the class of all collections of coverings on objects $X$ in $\mathfrak{_RAlg}^{op}$, we denote this $B(\mathfrak{_RAlg}^{op})$.  

\end{defn}

Since the Zarsiky coverings form a bassis for a Grothendieck Topology on $_RAlg^{op}$, it is not difficult to see that:
\begin{prop}
$B(\mathfrak{_RAlg}^{op})$ forms a Basis for a Gorthendieck topology on $\mathfrak{_RAlg}^{op}$.  
\end{prop}

\begin{defn}
\textbf{$\varsigma$-relative Zariski Site}

The pair of the category $\mathfrak{_RAlg}^{op}$ and the Grothendieck Topology generated by the $\varsigma$-relative Zariski Basis on $\mathfrak{_RAlg}$, is called the \textbf{$\varsigma$-relative Zariski Site}.  In keeping with classical tradition, we denote this site by $\mathfrak{_RAlg}^{op}_{\varsigma-Zar}$.  
\end{defn}

Now that we have topologized our \textit{"affine schemes-like objects"} in a way that it exactly the Zariski site on their largest $^{op}\circ G\circ^{op}$-relative affine part.  We glue them together, and select the $\varsigma$-relative Zariski sheaves who's largest $^{op}\circ G\circ^{op}$-relative scheme part can be covered by $^{op}\circ G\circ^{op}$-relative affines.  Essentially we're asking that our $^{op}\circ G\circ^{op}$-schemes be \textit{"at $^{op}\circ G\circ^{op}$-relative heart"}.  

First we make a little remark however,
\begin{prop}
For every object $A$ in $\mathfrak{_RAlg}$, the symbolic hom functor $Hom_{\mathfrak{_RAlg}}(A,-)$ defines a contravariant functor from $\mathfrak{_RAlg}$ to $\mathfrak{C}$.  In particular, there exist representable $\mathfrak{C}$-valued presheaves on $\mathfrak{_RAlg}^{op}$.  

Moreover, these representable $\mathfrak{C}$-valued presheaves are infact $\mathfrak{C}$-valued sheaves.
\end{prop}
\begin{proof}
The first statements are a direct consequence of the assumption that $\mathfrak{_RAlg}$ is $\mathfrak{C}$-enriched and the enriched Yoneda Lemma \textit{(contravariant statement)}.  

Similar arguments to those concerning the classical Zarisky site, imply that this $\varsigma$-relative Zariski Basis is subcanonical.  
\end{proof}
\begin{prop}
The functor $-^{op}\circ \varsigma \circ Y$ from the category of representable $\mathfrak{C}$-valued presheaves on $\mathfrak{_RAlg}^{op}$ to the category of affine schemes has a right adjoint, where $Y$ is the contravariant Yoneda embedding of $\mathfrak{_RAlg}$, into the category of representable $\mathfrak{C}$-valued presheaves on $\mathfrak{_RAlg}^{op}$. 
\end{prop}
\begin{proof}
Since $-^{op}$ and $Y$ are dualities, therefore in particualr they have right adjoints, by hypothesis $\varsigma$ has a right adjoint, therefore their composition \textit{(as remarked at the very beginnig of this paper)} has a right adjoint.  
\end{proof}
Let us denote the functor

We want our category of \textit{"scheme-like"} objects to contain enough affines, and for every scheme-like objects to admit a cover by \textit{"affine-like schemes"}; in a manner that the largest affine part of our 
schemes are exactly covered by the \textit{affine part} of the \textit{"affine-like cover"} in question.  

Tecal that in the classwical case it was enough for a functor from $R$-algebras to $Sets$ to be conside with a cover of affines on all $R$-algebra which were fields.  We use draw inspiration from this characterisation as follows.  

\begin{defn}
\textbf{$\varsigma$-Prescheme}

The category $\mathfrak{pCsh_R}$ of \textbf{$\varsigma$-Prescheme} is a category equivalent to a full-subcategory $\mathfrak{X}$ of the category of $\mathfrak{C}$-valued sheaves over the site $\mathfrak{_RAlg}^{op}_{\varsigma-Zar}$ such that the following hold:

- \underline{\textbf{$\varsigma$-Affines Existence Axiom}}

$\mathfrak{pCsh_R}$ contains a non-empty fullsubcategory $\mathscr{A}\mathfrak{pCsh_R}$ natrually equivalent to a fullsubcategory of $\mathfrak{X}$ consisting entirely of representable sheaves.  

- \underline{\textbf{Affine insides Axiom}}

For every object $A$ in $\mathscr{A}\mathfrak{pCsh_R}$, it's equivalent representable functor $Hom_{\mathfrak{_RAlg}}(A_i,-)$ has the property that For every $R$-algebra $\mathbb{K}$ that is a field there are the following containments of sets.  

$Hom_{_RAlg}(\varsigma(A), \mathbb{K}) \subseteq Forget(Hom_{\mathfrak{_RAlg}}(A,G(\mathbb{K})))$.  

\textit{(where $Forg$ is the forgetful functor from the concrete category $\mathfrak{C}$ to $Sets$)}. 

- \underline{\textbf{$\varsigma$-Affine Coverage Axiom}}

For every object $X$ in $\mathfrak{pCsh_R}$, there exists a small familly $\{ A_i \}$ of objects in $\mathfrak{_RAlg}$ such that for every object $\mathbb{K}$ in $_RAlg$ which is a field, the following equality holds:

$Forg( X(G(\mathbb{K})))=\underset{i}{\cup} Forg( Hom_{\mathfrak{_RAlg}}(A_i,G(\mathbb{K})))$.  
\end{defn}

\textit{As a general note, the name "Prescheme" seems fitting and has virtually noting to do with the classical notion of a prescheme as found in EGA.  Since the term has mostly fallen out of use it doesn't seem harmful to use it here.  }

\begin{defn}
\textbf{$\varsigma$-Affine Prescheme}

The objects of the category $\mathscr{A}\mathfrak{pCsh_R}$, described above are called \textbf{Affine $\varsigma$}.  
\end{defn}

However, our intensions are not only to simply approximate pointsets.  We wish to appriximate arrows in a totally fundamental way; that is we wished to approximate all our morphisms of these preschemes in a completly unique and unambigouly functorial manner.  To properly exibit our desired aproximatory behavior, we need to buildup a context in which we may easily pickout the objects and morphisms which fit our intentions.  This is the idea has all the while been the posibly most fundamental aspect of the entire framework, on which this entire paper was relies.  
\begin{defn}
\textbf{Algebra-geometric Context}

$
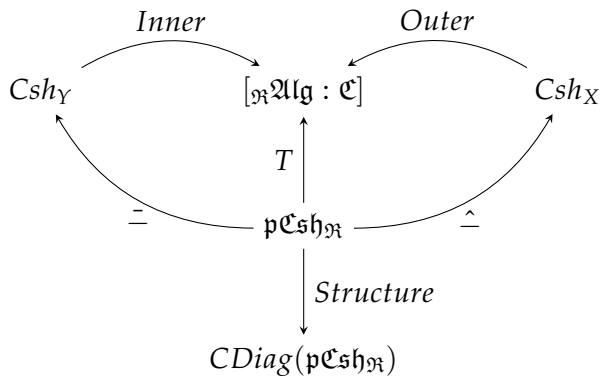
\begin{tikzpicture}
  \matrix (m) [matrix of math nodes,row sep=3em,column sep=4em,minimum width=2em]
  {   Csh_Y &  {[} \mathfrak{_RAlg}:\mathfrak{C} {]} & Csh_X \\
     {}  & \mathfrak{pCsh_R} & {}  \\
     {}  & CDiag(\mathfrak{pCsh_R}) & {} \\};
  \path[-stealth]
    (m-1-1) edge [bend left] node [above] {$Inner$} (m-1-2)
    (m-2-2) edge [bend left] node [below] {$\bar{-}$} (m-1-1)
    (m-1-3) edge [bend right] node [above] {$Outer$} (m-1-2)
    (m-2-2) edge [bend right] node [below] {$\hat{-}$} (m-1-3)
    (m-2-2) edge node [left] {$T$} (m-1-2)
    (m-2-2) edge node [right] {$ Structure $} (m-3-2);
\end{tikzpicture}
$

A \textbf{category of $\varsigma$-schemes} $\mathfrak{A}$ is $7$-tuple $<\mathfrak{pCsh_R}, Structure,Inner,Outer,\bar{-},\hat{-},\mathfrak{i}>$ where:

- $_{\mathfrak{A}}Csh$ is a category of $\varsigma$-preschemes.  

-$Outer: Csh_X \rightarrow [\mathfrak{_RAlg}:\mathfrak{C}]$, $Inner: Csh_Y \rightarrow [\mathfrak{_RAlg}:\mathfrak{C}]$, where $Csh_X$ nd $Csh_Y$ are categories of schemes above the schemes $X$ and above the scheme $Y$ respectivly, 

- $\bar{-}:_{\mathfrak{A}}Csh \rightarrow Csh_X$ 

- $\hat{-}: _{\mathfrak{A}}Csh\rightarrow Csh_Y$

- $\mathfrak{i}: \mathfrak{pCsh_R}\rightarrow [\mathfrak{_RAlg}:\mathfrak{C}]$ is the fully faithfull inclusion of the fullsubcategory $ $

Finally and most importantly,
Let $J$ be the category described by this graph \textit{(with identiy morpihsms ommitted)}
$
\begin{tikzpicture}
  \matrix (m) [matrix of math nodes,row sep=3em,column sep=4em,minimum width=2em]
  {
     1 & 2 & 3 \\
     4 & 5 & 6 \\};
  \path[-stealth]
    (m-1-1) edge node [left] {$ $} (m-2-1)
    (m-1-1) edge node [below] {$ $} (m-1-2)
    (m-2-1) edge node [below] {$ $} (m-2-2)
    (m-1-2) edge node [right] {$  $} (m-2-2)
    (m-1-3) edge node [right] {$ $} (m-2-3)
            edge node [below] {$ $ } (m-1-2)
    (m-2-3) edge node [below] {$ $} (m-2-2);
\end{tikzpicture}
$.  

- $Structure_f$ is a class of $J$-shapped diagrams indexed by the morphisms in $\mathfrak{pCsh_R}$, such that:
$Structure_f:J \rightarrow \mathfrak{pCsh_R}$, yielding the digram:

$
\begin{tikzpicture}
  \matrix (m) [matrix of math nodes,row sep=3em,column sep=4em,minimum width=2em]
  {
     Inner(\bar{A}) & A & Outer(\hat{A}) \\
     Inner(\bar{B}) & B & Outer(\hat{B}) \\};
  \path[-stealth]
    (m-1-1) edge node [left] {$Inner(\bar{f})$} (m-2-1)
    (m-1-1) edge node [below] {$\mu_0$} (m-1-2)
    (m-2-1) edge node [below] {$\mu_1$} (m-2-2)
    (m-1-2) edge node [right] {$ f $} (m-2-2)
    (m-1-3) edge node [right] {$Outer(\hat{f})$} (m-2-3)
            edge [->>] node [below] {$\epsilon_0$} (m-1-2)
    (m-2-3) edge [->>] node [below] {$\epsilon_1$} (m-2-2);
\end{tikzpicture}
$. 
\end{defn}
Before being able to describe the last step we introduce a special essentially already familiar category.

\begin{defn}
\textbf{$C(J,f)$}

Let $f$ be a morphism in $\mathfrak{Csh_R}$.  

The category $C(J,f)$ is the full subcategory of the category of $[J:\mathfrak{Csh_R}]$ of $J$-shaped diagram in $\mathfrak{Csh_R}$, who's objects are commutative diagrams taking the morphism $J$ to a \textit{commutative} diragram in $\mathfrak{pCsh_R}$ of the form:

$
\begin{tikzpicture}
  \matrix (m) [matrix of math nodes,row sep=3em,column sep=4em,minimum width=2em]
  {
     1 & 2 & 3 \\
     4 & 5 & 6 \\};
  \path[-stealth]
    (m-1-1) edge node [left] {$ Inner(g) $} (m-2-1)
    (m-1-1) edge node [below] {$ \mu_0 $} (m-1-2)
    (m-2-1) edge node [above] {$ \mu_1$} (m-2-2)
    (m-1-2) edge node [right] {$ f $} (m-2-2)
    (m-1-3) edge node [right] {$ Outer(h) $} (m-2-3)
            edge node [below] {$ \epsilon_0 $} (m-1-2)
    (m-2-3) edge node [above] {$ \epsilon_1 $} (m-2-2);
\end{tikzpicture}
$
, where $f$ is the aforemoentioned fixed arrow in $\mathfrak{Csh_R}$, with the morphisms $\mu_0$ and $\mu_1$ being monics and the morpihsms $\epsilon_0$ and $\epsilon_1$ are epis.  Moreover, the morphism $Inner(g)$ is the $Inner$-image of some morphism $g$ of schemes over $Y$ and $Outer(h)$ is the $Outer$-image of sme morphism $h$ of schemes over $X$.  
\end{defn}

The essence of our intent in approximating morphisms in \textit{\textbf{the} most accurate possible way by schemes from within and outwith}, is exactly the following final axiom which renders a $\varsigma$-scheme a workable construct.

\begin{defn}
\textbf{Category of $\mathscr{A}$-Schemes}

A \textbf{category of $\mathscr{A}$-schemes} is the \textit{(not necessarily full)} subcategory of the category of $\mathfrak{pCsh_R}$ in the geometric context $\mathfrak{A}$, who's every morphism $f$ satisfies the following crutial axiom:

- \underline{\textbf{$\mathscr{A}$-Approximability Axiom}}

The diagram $Structure(f)$ is an initial object in the category $C(J,f)$.  

We denote this category by $_{\mathscr{A}}\mathfrak{Csh}$ and call its objects \textbf{$\mathscr{A}$-schemes}.  In particualr we call those which are already $\varsigma$-Affine preschemes are now distinguished as, \textbf{$\mathscr{A}$-affine schemes}.  Lastly and most importantly, the morphism in this category are called \textbf{$\mathscr{A}$-morphisms} and the category is denoted by $_{\mathscr{A}Csh}$.  
\end{defn}

So in particular, if the identity morphism of an object fails this criterion, then that object cannot be in the category of $\mathfrak{A}$-schemes.  

Another way or reformulating this is as follows:
\begin{prop}
A morphism $f$ in $\mathfrak{_pCsh}$ satisfies the $\mathscr{A}$-approximation axiom if and only if it satisfies the following univeral property:

Keeping with the above notation let, $\mathfrak{A}$ be an algebro-geometric context, $\mu_0$ and $\mu_1$ are monics and $\epsilon_0$ and $\epsilon_1$ are epis in $\mathfrak{pCsh_R}$.  

If $g:X\rightarrow Y$ is a morphism in $Csh_X$, $h:Z\rightarrow W$ is a morphism in $Csh_Y$, and $\tilde{m_0}$ and $\tilde{m_1}$ are monics in $_{\mathfrak{A}}pCsh$, moreover $\tilde{e_0}$ and $\tilde{e_1}$ are epis in $_{\mathfrak{A}}Csh$ making the following diagram commute:

$
\begin{tikzpicture}
  \matrix (m) [matrix of math nodes,row sep=4em,column sep=2em,minimum width=2em]
  {
    Inner(X) & . & . &. & Outer(Z) \\
    . & Inner(\bar{A}) & A & Outer(\hat{A}) & .\\
    . & Inner(\bar{B}) & B & Outer(\hat{B}) & .\\
    Inner(Y) & . & . &. & Outer(W) \\};
  \path[-stealth]
    (m-2-2) edge node [left] {$Inner(\bar{f})$} (m-3-2)
    (m-2-2) edge node [below] {$\mu_0$} (m-2-3)
    (m-3-2) edge node [above] {$\mu_1$} (m-3-3)
    (m-2-3) edge node [right] {$ f $} (m-3-3)
    (m-2-4) edge node [right] {$Outer(\hat{f})$} (m-3-4)
            edge [->>] node [below] {$\epsilon_0$} (m-2-3)
    (m-3-4) edge [->>] node [above] {$\epsilon_1$} (m-3-3)
    (m-1-1) edge node [left] {$Inner(g)$} (m-4-1)
    (m-1-5) edge node [right] {$Outer(h)$} (m-4-5)
    (m-1-1) edge node [above] {$\tilde{m_0}$} (m-2-3)
    (m-4-1) edge node [below] {$\tilde{m_1}$} (m-3-3)
    (m-1-5) edge node [above] {$\tilde{e_0}$} (m-2-3)
    (m-4-5) edge node [below] {$\tilde{e_1}$} (m-3-3);
\end{tikzpicture}
$

Then there exist unique monics, $M_1$, $M_2$ in $_{\mathfrak{A}}Csh$ and unique epis $E_1$, $E_2$ in $_{\mathfrak{A}}Csh$ making the diagram commute:

$
\begin{tikzpicture}
  \matrix (m) [matrix of math nodes,row sep=4em,column sep=2em,minimum width=2em]
  {
    Inner(X) & . & . &. & Outer(Z) \\
    . & Inner(\bar{A}) & A & Outer(\hat{A}) & .\\
    . & Inner(\bar{B}) & B & Outer(\hat{B}) & .\\
    Inner(Y) & . & . &. & Outer(W) \\};
  \path[-stealth]
    (m-2-2) edge node [left] {$Inner(\bar{f})$} (m-3-2)
    (m-2-2) edge node [below] {$\mu_0$} (m-2-3)
    (m-3-2) edge node [above] {$\mu_1$} (m-3-3)
    (m-2-3) edge node [right] {$ f $} (m-3-3)
    (m-2-4) edge node [right] {$Outer(\hat{f})$} (m-3-4)
            edge [->>] node [below] {$\epsilon_0$} (m-2-3)
    (m-3-4) edge [->>] node [above] {$\epsilon_1$} (m-3-3)
    (m-1-1) edge [dashed] node [below] {$M_0$} (m-2-2)
    (m-4-1) edge [dashed] node [above] {$M_1$} (m-3-2)
    (m-1-5) edge [dashed] node [below] {$E_0$} (m-2-4)
    (m-4-5) edge [dashed] node [above] {$E_1$} (m-3-4)
    (m-1-1) edge node [left] {$Inner(g)$} (m-4-1)
    (m-1-5) edge node [right] {$Outer(h)$} (m-4-5)
    (m-1-1) edge node [above] {$\tilde{m_0}$} (m-2-3)
    (m-4-1) edge node [below] {$\tilde{m_1}$} (m-3-3)
    (m-1-5) edge node [above] {$\tilde{e_0}$} (m-2-3)
    (m-4-5) edge node [below] {$\tilde{e_1}$} (m-3-3);
\end{tikzpicture}
$

\end{prop}
\begin{proof}
Complete tautology.  
\end{proof}

\subsection{Examples}
Lets check some examples and to see that everything is inline:

Since a ring may be viewed as a category, then its dual, that is the affine scheme $Spec(R)$ may be interepreted as the dual category of $R$.  With this little remark in mind we indeed verify that:

\begin{ex}
\textbf{Classical Schemes}

Schemes over $Spec(R)$ are $\mathscr{A}_{Csh_R}$-schemes for the essentially trivial algebro-geometric context:

$\mathscr{A}_{Triv}: <Csh_{Spec(R)},Structure,1_{Csh_{Spec(R)}},1_{Csh_{Spec(R)}},1_{Csh_{Spec(R)}},1_{Csh_{Spec(R)}},i>$.  

Where $i$ is the ususal inclusion of $[_RAlg^{op}_{Zar}:Sets]$, where clearly Sets is concrete and $R$, and so $_RAlg$ are $Sets$-enriched.
\end{ex}

\begin{ex}
\textbf{Classical Group Schemes}

Group-Schemes over $Spec(R)$ are $\mathscr{A}_{GrpCsh_R}$-schemes for the algebro-geometric context:

$\mathscr{A}_{Triv}: <Csh_{Spec(R)},Structure,1_{Csh_{Spec(R)}},1_{Csh_{Spec(R)}},1_{Csh_{Spec(R)}},1_{Csh_{Spec(R)}},i>$.  

Where $i$ is the inclusion of $[_RAlg^{op}_{Zar}:Grp]$, where clearly the category $Grp$ of groups is concrete and $R$ and therefore $_RAlg$ are preadditive, in otherwords they are $Grp$-enriched.
\end{ex}

Part $3$ proved the following which we now state as a thrm in memoriam.

\begin{thrm}
\textbf{Nc. Schemes}

Nc. Schemes over $R$ are $\mathscr{A}_{Nc. Csh_R}$-schemes for the algebro-geometric context:

$\mathscr{Nc}: <\mathfrak{Csh_R},Structure,-_!,\mathfrak{P},\bar{-},\hat{-},i>$.  

Where $i$ is the inclusion of $[_R\mathfrak{Alg}^1$ $^{op}:Sets]$, with $_R\mathfrak{Alg}^1$ being the category of unital associative algebras topologised with the essential Zarsiky Topology.  

Note also that, the categories $Sets$, $R$ and $_RAlg$ are concretes, in otherwords they are $Sets$-enriched.
\end{thrm}
\begin{proof}
The entire Part $2$ and parts of Part $1$ of this text.  
\end{proof}
Since this is the central construct of this paper, lets name it.  

\begin{defn}
\textbf{Nc. context}

The algebro-gemoetric context $\mathscr{Nc}$, described above is called the \textbf{Nc. Conotext}.  
\end{defn}

From the above we need to only remark that the categories $R$ and therefore $_RAlg$ are also preadditive, in otherwords they are $Grp$-enriched to immediatly conclude the existence of Nc. group schemes.  

\begin{cor}
\textbf{Nc. Group Schemes}

Nc. Schemes over $R$ are $\mathscr{A}_{Nc. Csh_R}$-schemes for the algebro-geometric context:

$\mathscr{A}_{Nc. Csh}: <\mathfrak{Csh_R},Structure,-_!,\mathfrak{P},\bar{-},\hat{-},i>$.  

Where $i$ is the inclusion of $[_R\mathfrak{Alg}^1$ $^{op}:Grp]$, with $_R\mathfrak{Alg}^1$ being the category of unital associative algebras topologised with the essential Zarsiky Topology.  

Note also that, the category Sets is concrete and $R$ therefore $_RAlg$ are also preadditive, in otherwords they are $Grp$-enriched.
\end{cor}

If we infact wish, we may particularise this example further to reap homological algebraic gain:

\begin{cor}
\textbf{Nc. Abelian Group Schemes}

Nc. Schemes over $R$ are $\mathscr{A}_{Nc. Csh_R}$-schemes for the algebro-geometric context:

$\mathscr{A}_{Nc. Csh}: <\mathfrak{Csh_R},Structure,-_!,\mathfrak{P},\bar{-},\hat{-},i>$.  

Where $i$ is the inclusion of $[_R\mathfrak{Alg}^1$ $^{op}:Ab]$, with $_R\mathfrak{Alg}^1$ being the category of unital associative algebras topologised with the essential Zarsiky Topology.  

Note also that, the category Sets is concrete and $R$ therefore $_RAlg$ are also preadditive, in otherwords they actualy more that $Grp$-enriched, they are infact enriched over abelian groups.
\end{cor}

Since abelian groups are groups, which are groupoids we may generalize the above to obtain:

\begin{cor}
\textbf{Nc. Group Schemes}

Nc. Schemes over $R$ are $\mathscr{A}_{Nc. Csh_R}$-schemes for the algebro-geometric context:

$\mathscr{A}_{Nc. Csh}: <\mathfrak{Csh_R},Structure,-_!,\mathfrak{P},\bar{-},\hat{-},i>$.  

Where $i$ is the inclusion of $[_R\mathfrak{Alg}^1$ $^{op}:Grpd]$, with $_R\mathfrak{Alg}^1$ being the category of unital associative algebras topologised with the essential Zarsiky Topology.  

Where $Grpd$ is the category of Groupoids and morphisms of Groupoids.  
\end{cor}

\section{Abstract Approximation Theorem}
WE'll need this which we sketch only.  
\begin{prop}
In an algebro-geometric context $\mathscr{A}$, the category $[\mathfrak{_RAlg}L\mathfrak{C}]$ satisfies $WISC$.  
\end{prop}
\begin{proof}
By construction the covering on $\mathfrak{_RAlg}^{op}$ is essentially identical to the Zarisk covering on $_RAlg^{op}$, since the category of scheves on the later satifies WISC, so must the category of scheves of the former, after post-composing with the forgetful functor $\mathfrak{C}\rightarrow Sets$.  

\end{proof}

In the proof of the the nc. approximation theorem, we notice that no specific properties of NC where used, therefore we may completly translate the result into the setting of an arbitrary algebro-geometric context; the proof following \textit{mutatis mutandis}.  

\textit{(Keeping in line with the above established notation)}

\begin{thrm} $\label{AbsThrm}$
\textbf{Abstract Approximation Theorem}
If $f$ is a morphism of commutative schemes both over the schemes $X$ and $Y$ and \underline{\textbf{P}} is a category independent property of that morphism, then both $\Lambda(Free_{\star} \circ Inner(f))$ and $\Lambda(Free_{\star} \circ Outer(f))$ have the category independent property \underline{\textbf{P}}.  

\textit{(where $Free$ is the left adjoint to the forgetful functor from $\mathfrak{C}$ to sets, and $Free_{\star}:[S^{op}:Sets]\rightarrow [S^{op}:\mathfrak{C}]$ is the post-composition with it).  }
\end{thrm}

\begin{ex}
A familiar example, is that a morphism of group schemes is smooth if and only if it is smooth when viewed as a morphism of schemes.  
\end{ex}

With that food for thought we end this philosophical portion of the paper.  
\subsubsection*{ }
\textit{Thank you very much for your time, I cincearly hope you enjoyed reading this atleast as much as I enjoyed working on it.  }

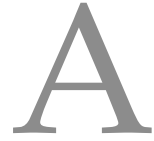
\appendix
\pdfbookmark[1]{References}{References}
\chapter{References}
\section*{Monographs}


The Principal refernce has been:

-\textbf{Abrams, L, Weibel C.}.  \textit{Cotensor products of modules}. Arxiv. (1999), Web.  

-\textbf{Arapura, D.}.  \textit{Algebraic Geometry over the complex numbers}.  Springer, 2012. Print.  

-\textbf{Borel, Armand}.  \textit{Linear Alebraic Groups}.  New Jersey: Princeton University Press  (1991). Web.  

-\textbf{Broer, Abraham}.  \textit{Introduction to Commutative Algebra}.  Montreal: Université de Montréal (2013). Web.  

-\textbf{Cameron, Peter J.}.  \textit{Notes on Classical Groups}.  London: School of Mathematical Sciences (2000). Web.  

-\textbf{Cartan, H. and Eilenberg, S.  }.  \textit{Homological algebra}.  New Jersey.  Princeton University Press, 1956. Print.    

-\textbf{Cibotaru, D.}.  \textit{Sheaf Cohomology}.  Indiana.  Notre Dame University, 2005. Web.    

-\textbf{J. Cuntz, D. Quillen}.  \textit{Algebra extensions and nonsingularity}, AMS (1995), Web. 

-\textbf{Cuntz, J, Skandalis, G, Tsygan, B.}.  \textit{Cyclic Homology in \\
Non-Commutative Geometry}, Springer (2004), Print. 

-\textbf{Demazure, M., Gabriel P}.  \textit{Introduction to Algebraic Geometryand Algebraic Groups}. Ecole Polytechnique, France (1980), Print.  

-\textbf{Eilenberg, Samuel; Mac Lane, Saunders}.  \textit{On the groups of H($Pi$,n). I}.  Annals of Mathematics. (1953).  Print.  

-\textbf{Gelfand, Manin}.  \textit{Methods of Homological Algebra, 2nd ed.}.  Springer, 2000.  Print.  

-\textbf{Godement, R.} \textit{Topologie algébrique et théorie des faisceaux}.  Paris, Hermann, 1973.   Print.  

-\textbf{Grothendieck, A.}.  \textit{Éléments de géométrie algébrique (I-III)}. Springer, 1971.  Print. \\ -\textbf{Grothendieck, A.}.  \textit{Sur quelques points d'algèbre homologique}.  The Tohoku Mathematical Journal, 1957.  Print.  

-\textbf{Guccione, J. A \textit{and} Cuccione J. J.}.  \textit{The theorem of excision for Hoschild and cyclic homology}.  Journal of Pure and Applied Algebra 106.  Buenos Ares. Argentina, 1996.  Web.  

-\textbf{Harari, David}.  \textit{Schemes}.  Beijing: Tsinghua University (2005). Print.  

-\textbf{Hartshorne, Robin}.  \textit{Algebraic Geometry}.  New-York: Springer (1977). Print.  

-\textbf{G. Hochschild}.  \textit{On the cohomology groups of an associative algebra}, J. Amer. Math. Soc. (1945), Web. 

-\textbf{Humphreys, James E.}.  \textit{Introduction to Lie Algebras and Representation Theory}. New York: Springer-Verlag, (1972). Print.  

-\textbf{F. Ischebeck}.  \textit{Eine Dualitat zwischen den Funktoren Ext und Tor}, J. Algebra 11 (1969), Web. 

-\textbf{Jantzen, Jens Carsten}.  \textit{Representations of Algebraic Groups}.  London: Academic Press Inc. (London) LTD. (1987). Print.  

-\textbf{Johnson S.}.  \textit{Summary: Analytic Nullstellensatz Pt 1.}.  Vancoover. 2012. Web.

-\textbf{Kleshchev, Alexander}.  \textit{Lectures on Algebraic Groups}.  Oregon: University of Oregon (2008). Web.  

-\textbf{Knapp, Anthony W.}.  \textit{Lie Groups Beyond an Introduction - 2nd Edition}.  Boston: Birkhäuser (2004). Print.  

-\textbf{J.C. McConnell, J.C. Robson}.  \textit{Noncommutative Noetherian rings}, Ann. of Math. (1945), Web.

-\textbf{Milne, James}.  \textit{Basic Theory of Affine Group Schemes (AGS)}.  (2012) Web.  

-\textbf{AL-Takhman, Khaled}.  \textit{Equivalences of Comodule Categories for
Coalgebras over Rings}, (2001), Web.  

-\textbf{T.Y. Lam}.  \textit{Lectures on modules and rings}, Springer-Verlag (1999), Web. 

-\textbf{Loday, J.L.}.  \textit{Cyclic Homology}.  Berlin Heidenberg. Springer-Verlang.  (1998).  Print.  

-\textbf{D. Quillen}.  \textit{Algebra cochains and \\
cyclic cohomology}, Inst. Hautes Etudes Sci. Publ. Math. 68 (1988), Web. 

-\textbf{Royden H., Fitzpatrick P.}.  \textit{Real Analysis - 4th Edition}.  Boston: Pearson Hall (2010). Print.  

-\textbf{W. F. Shelter}.  \textit{Smooth algebras}, J. Algebra 103, (1986), Web. 

-\textbf{Springer, Tony A.}.  \textit{Linear Algebraic Groups}.  Budapestlaan: Modern Birkhäuser (2008). Print.  

-\textbf{M. van den Bergh}.  \textit{A relation between Hochschild homology and cohomology for Gorenstein rings}. Proc. Amer. Math. Soc. 126 (5) (1998), 1345-1348. Erratum: Proc. Amer. Math. Soc. 130 (9) (2002), 2809-2810 (electronic).  Web. 

-\textbf{C. Weibel}.  \textit{An introduction to homological algebra}, Cambridge University Press (1995), Web. 

-\textbf{Wodzicki, M. }.  \textit{Excision in Cyclic Homology and in Rational Algebraic K-theory}. Annals of Mathematics 129 (1989), 591-639.  Web.  

-\textbf{C. Weibel}.  \textit{An introduction to homological algebra}, Cambridge University Press (1995), Web. 

\tiny{\tiny{ICXRNIKA}}

\chapter{An example: A new class of geometric objects}
Now that nc. schemes and their proper properties have been motivated, we introduce a new type of geomtrtic object which is \textit{neither} a scheme \textit{(as viewed in $[_R\mathfrak{Alg}:Sets]$)} nor a nc. scheme.  Infact it is a non-trivial mixture of the two, fused together.  

I introduce these objects, since they take substantial advantage of the geometric objects in the category $[_R\mathfrak{Alg}:Sets]$, pieced together both from, purely nc. geometric and \textit{"commutative"} geometric objects.  

First however we recall a common bifunctor from $[_R\mathfrak{Alg}:Sets]$ and $[_R\mathfrak{Alg}:Sets]$ to $[_R\mathfrak{Alg}:Sets]$.  

\begin{defn}
\textbf{Interesection of functors}

If $X$ and $Y$ are functors in $[_R\mathfrak{Alg}:Sets]$, then their interesection $X \cap Y$ is the functor in $[_R\mathfrak{Alg}:Sets]$ defined on $R$-algebra morphisms $f:A \rightarrow B$ as $X \cap Y (f) := X(f)|_{X(A)\cap Y(A)} \cap Y(f)|_{X(A)\cap Y(A)}$. 
\end{defn}
It is not difficult to see that this is indeed a functor since the target category is $Sets$.  
\begin{lem}
If $X$ and $Y$ are functors in $[_R\mathfrak{Alg}:Sets]$, then there exist a unique split monic $i: X\cap Y \rightarrow X$.  With the universal properties: 

-if $f:A\rightarrow X$ is a morphism in $[_R\mathfrak{Alg}:Sets]$ then there exists a unique morphism $j:A \rightarrow X \cup Y$ such that $j\circ i = f$ 

-$g:X\rightarrow B $ is a morphism in $[_R\mathfrak{Alg}:Sets]$ then there exists a unique morphism $h:X\cap Y \rightarrow A$ such that $g\circ i = h$.  
\end{lem}
\begin{proof}
Cosnider the split monic given $i: X\cap Y \rightarrow X$ given by the familly of inclusions of sets $X(A)\cap Y(A)$ into $X(A)$.  Since inclusions are split monic in $Sets$ then $i$ inherits this and is split monic in $[_R\mathfrak{Alg}:Sets]$.  

Since $i$ is \textit{split} monic, the first universal property comes from the uniquness of the inverse left inverse $i^{-1}$ of $i$, setting $j$ to be the therefore uniquely determined morphism $i^{-1}\circ f$.  

Similarly, $i$ is \textit{monic} so the first universal property comes from setting $h$ to be the uniquely determined morphism $f\circ i$.  
\end{proof}
We'll call this universal arrow $i: X \cap Y \rightarrow X$ the inclusion of $X\cap Y$ in $X$.  

We are now ready.  
\begin{defn}
\textbf{Pollock Space}

A \textbf{Pollock Space over $R$} $\mathscr{P}$ is an ordered pair of morphisms $<c,n>$ in $[_R\mathfrak{Alg}:Sets]$ with common codomain, such that $c$ and $n$ is the \textit{pushout} of the inclusions $i: X \cap Z \rightarrow X$ and $j: X\cap Z \rightarrow Z$ \textit{respectivly}; where $X$ is naturally isomorphic to $Y_!$, where $Y$ is some scheme and $Z$ is isomorphic to a nc. scheme.  
\end{defn}
\begin{defn}
\textbf{Commutative and Nc. Parts of a Pollock Space}
If $P:=<c,n>$ is a Pollock space, then $c$ is called the \textbf{commutative part of $P$} and $n$ is called the \textbf{nc. part of $P$}.  
\end{defn}

Directly from the definition we have:
\begin{prop}
If $X$ is a scheme and $Z$ is a nc. scheme then there is a \textit{unique up to natural isomorphism} Pollock space $P:=<c,n>$, such that $X$ and $Z$ are both subfunctors of $P$.  Moreover, if $X$ and $Z$ are both subfunctors of a functor $F$ in $[_R\textit{Alg}:Sets]$, then there exists:

-A unique natural transformation $u:P\rightarrow F$, such that $l = u\circ c$ and $m = v\circ n$, where $l:X \rightarrow F$ and where $m:Z \rightarrow F$ are the unique split monics arising from the subfunctor relations between $X$ and $Z$ with respect to $F$, respectivly.  
\end{prop}
\begin{proof}
Pushouts always exist in an elementary topoi, giving existence, the rest is just a rephrasing of the universal property of a pushout.  
\end{proof}

Another tautological result is the following remark.  

Finally, we give an example of an al ltogether exotic Pollock Scheme.  

\begin{ex}
Let $A$ and $B$ be distinct commutative unital associative $R$-algebras, then the pushout $Y(A)\star Hom_{_RAlg}(B,-)_!$ is a Pollock scheme.  
\end{ex}

Mixed affine spaces:
\begin{ex}
Let $I$ and $J$ be indexing sets with $J$ baing of strictly greater cardinality that $I$. 

Now consider the $R$-algebras $R<X_i>_{i\in I}$ and $R[X_i]_{i\in J}$, then the pushout $Y(R<X_i>_{i\in I})\star Hom_{_RAlg}(R[X_i]_{i\in J},-)_!$ is a Pollock scheme, which is neither a nc. scheme nor a scheme \textit{(as viewed withing $[_R\mathfrak{Alg}:Sets]$ via the functor $-_!$)}.  
\end{ex}

Mixed projectives:
\begin{ex}
Let the basering be the complex field $\mathbb{C}$, and $X_1$ and $X_2$ be distinct projective schemes.  

Then the pushout $\Lambda(\mathfrak{P}(X_1)) \star (X_2)_!$ is a Pollock Space.  Which again, is neither a nc. scheme nor a scheme as viewed in $[_R\mathfrak{Alg}:Sets]$ via the functor $-_!$.  
\end{ex}

\subsection{Interpretations of morphisms of Pollock Spaces}
The universal property of a pushout then implies that a morphism of Pollock spaces must an ordered pair of morphisms of $\phi:=<f_!,g>$ where $f$ is a morphism of schemes and $g$ is a morphism of nc. schemes.  Keeping with our little tradition, we say that $f_!$ is the commutative part of $\phi$ and $g$ is the nc. part of $\phi$.  

We may view a that a morphism of \textit{Pollock spaces} $\phi$ as having an analogue to a category independent property \underline{\textbf{P}}, which we call \textit{Pollock \underline{\textbf{P}}} \textit{if and only if} both, the commutative part of $\phi$ has the category independent property \underline{\textbf{P}} and the nc. part of $\phi$ has the property Nc. \underline{\textbf{P}}.  

\begin{prop}
We would have that a morphism $\phi:=<c:X\rightarrow Y, n: Z\rightarrow W>$ of Pollock spaces is \textit{Pollock smooth} if and only if all the three sheaves of differentials $\Omega_{Y|Z}$, $\Omega_{\bar{W}|\bar{Z}}$ and $\Omega_{\hat{W}|\hat{Z}}$ are locally free.  
\end{prop}
\begin{proof}
This follows directly from the defintions, of nc. smoothness, smoothness and the characterisation as a morphism of schemes $f:X \rightarrow Y$ being smooth if and only if the sheaf of differentials $\Omega_{Y|X}$ is locally free.  
\end{proof}

With this little food fro thought we end this paper.

\chapter{Abstractions} $\label{ch:Abs}$
\subsection*{Foreword}
We now take a very brief transitional repause, in order to abstract the machinery developed in the past section so that it may be applicable to a host of different settings in a multitude of different ways.  

Once this is quick process is completed, we exhibit a very beautiful geometric interpretation of essentially smooth, essentially etale and essentially unramified'significance.

\section{Abstract Concepts}
\begin{defn}
\textbf{F-Essentially Smooth}

IF $F:\mathfrak{C} \rightarrow \mathfrak{D}$ is a functor, then morphism $f$ in $\mathfrak{C}$ is said to be \textbf{essentially $F$-smooth} \textit{if and only if} the morphism $F(f)$ is formally smooth in $\mathfrak{D}$.  

\end{defn}

\begin{defn}
\textbf{F-Essentially Unramified}

IF $F:\mathfrak{C} \rightarrow \mathfrak{D}$ is a functor, then morphism $f$ in $\mathfrak{C}$ is said to be \textbf{essentially $F$-unramified} \textit{if and only if} the morphism $F(f)$ is a formally unramified morphism in $\mathfrak{D}$.  

\end{defn}

\begin{defn}
\textbf{F-Essentially Etale}

IF $F:\mathfrak{C} \rightarrow \mathfrak{D}$ is a functor, then morphism $f$ in $\mathfrak{C}$ is said to be \textbf{essentially $F$-etale} \textit{if and only if} the arrow $F(f)$ is formally etale in arrow $\mathfrak{D}$.  

\end{defn}

\section{An Abstract Theorem}
Theorem $\autoref{thrm:lethrm}$ may be presented in its fully abstract form, demonstrating how the idea of $F$-essential smoothness, $F$-essential etale and $F$-essentially unramified does infact provide a strict generalisation of their respective formal counterparts.  Specifically, the corrolarry presented directly after exposes a way how, for example the concept of essential smoothness can be used to \textit{"push-formawrd"} notions of formall smoothness from a subcategory of a category on which the notion is well behaved to the entire category on which these formal constructs may act unfavorably.  

\begin{thrm} $\autoref{thrm:abstrEt}$
If $F$ is a faithful endofunctor on a category $\mathfrak{C}$, then a morphism $f$ in $\mathfrak{C}$ is:

-$F$-essentially smooth if and only if it is smooth.  

-$F$-essentially unramified if and only if it is unramified.  

-$F$-essentially etale if and only if it is etale.  

\end{thrm}

\begin{proof}
Simliar to the idea of the proof of theorem $\autoref{thrm:lethrm}$.  
\end{proof}

Any inclusion of a category into another, that is anuy forgetful functor provides a straightforward example of this construction.  
\begin{cor}
If $F$ is a forgetfulfunctor, then a morphism $f$ in $Dom(F)$ is:

-$F$-essentially smooth if and only if it is smooth.  

-$F$-essentially unramified if and only if it is unramified.  

-$F$-essentially etale if and only if it is etale.

\end{cor}

\part{Associated Topological Space}
\chapter{Essentially Etale Topos over a nc. scheme}
\subsubsection{Foreword}
Our final objective is to ascribe an actual topological space to a nc. scheme over $\mathbb{C}$.  To manage this, we construct and exploit an adaptation of Grothendeick's Comparison Theorem, a refiniement of our essentially Zariski topology to one which is fppf on the commutative part of a nc. scheme and we concclude by modifying the Nc.-approximability Axiom of the Nc. Geometry context to suit this scenario.  

\section{Essentially fppf Topos}

We proceed similarly to our earlier constructions.  

\begin{defn}
\textbf{Essentially-\textit{fppf} Morphism}

A morphism $\phi$ in the category $\mathfrak{_RAlg^1}^{op}$ is said to be \textbf{Essentially-fppf} if and only if the morphism of affine schemes $(((\phi)^{op})^{\varsigma})^{op}$ is:

- flat and

- of fintie local presentation.  

\end{defn}

Following our geometric intuition, that these morphism cover our objects and behave as an fppf no the \textit{commutative part} of our nc. affine schemes we define an \textit{essentially fppf covering} as follows:

\begin{defn}
\textbf{Essentially-\textit{fppf} Covering}

A small familly of morphisms $\{ \phi_i: Y_i \rightarrow X \}$ is an \textbf{essentially-fppf covering} \textit{if and only if} 

$\underset{i}{\cup} \phi_i(Y_i)^{\varsigma} = X^{\varsigma}$.  
\end{defn}

\begin{defn}
\textbf{Essentially-\textit{fppf} Site}

The pair of the category $\mathfrak{_RAlg^1}^{op}$ and the coverage consisting of all $fppf$-coverings therein is called the \textbf{essentially-fppf site} and is denoted \textbf{$\mathfrak{_RAlg}_{Efppf}$}.  
\end{defn}
We should make one important verififcation.  
\begin{prop}
$\mathfrak{_RAlg}_{Efppf}$ is infact a site.  
\end{prop}

Now that our ingredients are inplace we're ready too cookup something interesting.  

\begin{defn}
\textbf{Category of Essentially-Algebraic Spaces}

The fullsubcategory of the category of $Sets$-valued sheaves on the Site $\mathfrak{_RAlg}_{Efppf}$ of objects $X$ satisfying the following two properties:

- $\Delta (X (-^{\varsigma})): X(-^{\varsigma}) \rightarrow X(-^{\varsigma}) \times X(-^{\varsigma})$ is a representable natural transformation.  

- There exists a representable functor $U\rightarrow X$, from an affine scheme $U$ and a surjective etale morphism of sheaves $U \rightarrow \Lambda_{et} (X(-^{\varsigma}))$ \textit{(where $\Lambda_{et}$ is the sheafification functor for the etale topos)}.  
\end{defn}

\begin{prop}
Nc. schemes are essentially-algebraic spaces.  
\end{prop}

\begin{proof}

The standardization of an nc. affien scheme is an affine scheme, then the diagonal morphism is representable. Therefore, every nc. affine scheme is an essentially algebraic space.  

Since, The standardizxation of a morphism is a zariski-immersion only if it is an fppf-immersion.  In other words every sheaf that can be covered by nc. affine scheme is a sheaf in the \textit{essentially fppf}-topos and is an essentially-algebraic space; in particualr this is so for nc. schemes.   
\end{proof}
    
\end{document}